\documentclass[letterpaper]{amsart}
\usepackage{amssymb}
\usepackage{amsmath}
\usepackage{amsfonts}

\newtheorem{theorem}{Theorem}
\theoremstyle{plain}
\newtheorem*{acknowledgement}{Acknowledgement}

\newtheorem{corollary}[theorem]{Corollary}

\newtheorem{definition}[theorem]{Definition}

\newtheorem{lemma}[theorem]{Lemma}

\newtheorem{proposition}[theorem]{Proposition}
\newtheorem{remark}[theorem]{Remark}

\numberwithin{equation}{section}
\input{tcilatex}
\begin{document}
\title{Crossed products and cleft extensions for coquasi-Hopf algebras}
\author{Adriana BALAN}
\address{Department of Mathematics, \\
Faculty of Applied Sciences, University Politehnica of Bucharest\\
313 Splaiul Independen\c{t}ei, 060042 Bucharest, Rom\^{a}nia}
\email{asteleanu@yahoo.com}
\date{October, 2008}

\begin{abstract}
The notion of crossed product by a coquasi-bialgebra $H$ is introduced and
studied. The resulting crossed product is an algebra in the monoidal
category of right $H$-comodules. We give an interpretation of the crossed
product as an action of a monoidal category. In particular, necessary and
sufficient conditions for two crossed products to be equivalent are
provided. Then, two structure theorems for coquasi Hopf modules are given.
First, these are relative Hopf modules over the crossed product. Second, the
category of coquasi-Hopf modules is trivial, namely equivalent to the
category of modules over the starting associative algebra. In connection the
crossed product, we recall from \cite{Balan07}\ the notion of a cleft
extension over a coquasi-Hopf algebra. A Morita context of Hom spaces is
constructed in order to explain these extensions, which are shown to be
equivalent with crossed product with invertible cocycle. At the end, we give
a complete description of all cleft extensions by the non-trivial
coquasi-Hopf algebras of dimension two and three.
\end{abstract}

\maketitle

\section{Introduction}

The notion of a crossed product by a bialgebra was first introduced by
Sweedler, in his study on cohomology over bialgebras (\cite{Sweedler68}).
Later it was generalized and intensively studied in relation with the theory
of Hopf-Galois algebra extensions (\cite{Blattner86}, \cite{Blattner89}, 
\cite{Doi89a}, \cite{Doi86}). As it was noticed in \cite{Schauenburg03}, a
crossed product by a bialgebra $H$ can be interpreted as an action of the
monoidal category of comodules over the category of modules over a ring.
Using this point of view, it is natural to try to define crossed product by
a coquasi-bialgebra instead of a bialgebra. This is motivated by the fact
that coquasi-bialgebras generalize bialgebras, preserving the monoidality of
the category of comodules. One of the aims of this paper is to introduce the
coquasi-algebraic version of crossed product by a bialgebra, including
interpretations in terms of monoidal categories, and also to study some
properties of such crossed products.

In the Hopf algebra theory, crossed products are the same as cleft
extensions (\cite{Doi86}, \cite{Blattner89}). Motivated by this
correspondence, we also investigate the notion of a cleft extension over a
coquasi-Hopf algebra (which was already introduced by the author in the
previous paper \cite{Balan07}).

The paper is organized as follows. It begins with a short review of the
known results about coquasi--bialgebras and coquasi-Hopf algebras, their
categories of comodules and about algebras and modules within these monoidal
categories mentioned above. In Section 3.1, we shall see that given a
coquasi-bialgebra $H$\ and an associative algebra $R$, endowed with a weak
left $H$-action, the existence of a $R$-valued 2-cocycle allows us to define
a multiplication on $R\otimes H$, by the same formula as in the Hopf algebra
situation. But the similarity stops here: the conditions that we have to
impose on the cocycle are modified because of the reassociator $\omega $ of $%
H$. The resulting crossed product will no longer be an associative algebra,
but a right $H$-comodule algebra $R\overline{\#}_{\sigma }H$. Also it is
interesting to notice that the crossed product can be built on the base
field $\Bbbk $ if and only if the coquasi-bialgebra is a deformation of a
bialgebra. An example of crossed product is provided by the associative
algebra $R=H^{\ast }$, with regular left weak action and cocycle given by
the reassociator $\omega $. Then $H^{\ast }\overline{\#}_{\sigma }H$ can be
interpreted as an analogue of the Heisenberg double for coquasi-bialgebras
(the quasi-bialgebra case was studied by Panaite and Van Oystaeyen in \cite%
{Panaite04}). This can be generalized by taking $R=Hom(H,A)$, for any $A$ a
right $H$-comodule algebra. Then there is an associative multiplication on
this space, which generalize Doi's smash product and allows us to construct
the crossed product $Hom(H,A)\overline{\#}_{\sigma }H$. Another example is
obtained in the finite dimensional case. Namely, we show that giving an
associative algebra $R$ together with a weak action and a two cocycle such
that relations (\ref{weak action}), (\ref{unit}), (\ref{asociat}), (\ref%
{vanish cocycle}), (\ref{cocycle}) hold (a crossed system), is the same as
giving a right $H^{\ast }$-comodule algebra (as it was defined by Hausser
and Nill in \cite{Hausser99}). In particular, the crossed product in this
case coincides with the quasi-smash product from \cite{Bulacu02}.

The next Section is devoted to find a categorical explanation of the
conditions imposed on the cocycle and weak action. Namely, the monoidal
category of $H$-comodules (or bicomodules) acts on the category of $R$%
-modules (or $R$-modules, $H$-comodules) by usual tensor product if the
conditions for the crossed product are fulfilled. Changing the monoidal
category by twisting the coquasi-bialgebra implies changing the crossed
system. Two structures of crossed product on the same algebra $R$ with same
coquasi-bialgebra $H$ (meaning we change the action of the monoidal
category) are equivalent if and only if the corresponding cocycles differ by
a coboundary.

In Section 3.3, the category of coquasi-Hopf modules $(\mathcal{M}%
_{R}^{H})_{H}$ is introduced and studied, again by monoidal category
arguments. Namely, the category of $H$-bicomodules $^{H}\mathcal{M}^{H}$
acts on the category of right $R$-modules, $H$-comodules $\mathcal{M}%
_{R}^{H} $ by usual tensor product and $H$ is an algebra in this monoidal
category. Hence it makes sense to construct right $H$-modules within $%
\mathcal{M}_{R}^{H}$. These will be called right coquasi-Hopf modules. Now
the crossed product algebra comes in: the category of right coquasi-Hopf
modules $(\mathcal{M}_{R}^{H})_{H}$ is isomorphic to the category of
relative Hopf modules over $R\overline{\#}_{\sigma }H$. It is interesting to
notice that a similar category, but for the finite dimensional dual case,
was defined in \cite{Bulacu02} and called the category of two-sided Hopf
modules. This is isomorphic to our category of coquasi-Hopf modules, but the
isomorphism seems to do not have a monoidal category explanation (Remark \ref%
{eu si bulacu module in trei colturi}). This Section ends with a structure
theorem for the category of coquasi-Hopf modules: we show that this category
is trivial, if the coquasi-bialgebra is endowed with an antipode and the
cocycle is invertible: namely, there is a special projection on the
coinvariants space, which induces an equivalence with the category of right
modules over the starting associative algebra $R$. Combinig this with
Theorem \ref{structura coquasihopf modulelor}, it follows that Hopf modules
over the crossed product algebra are trivial. In the Hopf algebra case, this
holds because crossed products with invertible cocycle are the same as cleft
extensions. The second main part of the paper is devoted to find a similar
result in the context of coquasi-Hopf algebras. But this requires an
appropriate notion of cleft extension for coquasi-Hopf algebras. Given a
right $H$-comodule algebra $A$, this is a cleft extension of the subalgebra
of coinvariants $B=A^{coH}$ if conditions (\ref{inversecleaving})-(\ref%
{convolutiegamabetadelta}) hold. This definition was introduced in author's
previous paper \cite{Balan07} and is significantly different from that of
cleft Hopf algebra extensions. As this involves the convolution product
(which is no longer associative), the invertibility of the cleaving map has
to be translated now in relations (\ref{convolutiedeltagama}), (\ref%
{convolutiegamabetadelta}) involving the antipode and the linear maps $%
\alpha $, $\beta $. We shall give an interpretation of these relations.
Namely, a Morita context involving four different $Hom$ spaces is
constructed, similar to the one used in \cite{Bohm07} for the coring case.
The strictness of the context is deeply connected with the notion of Galois
extension (as it was defined in \cite{Balan07}), and the cleftness is
equivalent to the existence of two elements in the connecting bimodules
which are mapped by the Morita homomorphisms to the units elements of the
involved algebras, in particular the Morita context is strict.

In a previous paper (\cite{Balan07}), we have shown the equivalence between
cleft extensions and Galois extensions with normal basis property.\ It is
then\ natural to pursue the characterization of cleft extensions in terms of
crossed products with coquasi-Hopf algebras. We generalize in Section 4.2
the result of Doi and Takeuchi (\cite{Doi86}), respectively of Blattner and
Montgomery (\cite{Blattner89}) about the equivalence between the two
structures mentioned previously. As an application of this, in the Appendix
we give a full characterization of all cleft extensions by certain
coquasi-Hopf algebras, namely the unique non-trivial coquasi-Hopf algebras
of dimension 2 and 3, as they were described in \cite{Albuquerque99a}.

As we shall see, the theory of coquasi-Hopf algebras is technically more
complicated than the classical Hopf algebra theory. This happens because of
the appearance of the reassociator $\omega $ and of the elements $\alpha $
and $\beta $ in the definition of the antipode. All these things increase
the complexity of formulas, and therefore of computations and proofs.

\section{Preliminaries}

In this Section we recall some definitions, results and fix notations.
Throughout the paper we work over some base field $\Bbbk $. Tensor products,
algebras, linear spaces, etc. will be over $\Bbbk $. Unadorned $\otimes $
means $\otimes _{\Bbbk }$. An introduction to the study of quasi-bialgebras
and quasi-Hopf algebras and their duals (coquasi-bialgebras, respectively
coquasi-Hopf algebras) can be found in \cite{Majid95}. A good reference for
monoidal categories is \cite{Kassel95}, while actions of monoidal categories
are exposed in \cite{Schauenburg03}.

\begin{definition}
A coquasi-bialgebra $(H,m,u,\omega ,\Delta ,\varepsilon )$ is a
coassociative coalgebra $(H,\Delta ,\varepsilon )$ together with coalgebra
morphisms: the multiplication $m:H\otimes H\longrightarrow H$ (denoted $%
m(h\otimes g)=hg$), the unit $u:\Bbbk \longrightarrow H$ (denoted $%
u(1)=1_{H} $), and a convolution invertible element $\omega \in (H\otimes
H\otimes H)^{\ast }$ such that: 
\begin{eqnarray}
h_{1}(g_{1}k_{1})\omega (h_{2},g_{2},k_{2}) &=&\omega
(h_{1},g_{1},k_{1})(h_{2}g_{2})k_{2}  \label{asociat multipl} \\
1_{H}h &=&h1_{H}=h \\
\omega (h_{1},g_{1},k_{1}l_{1})\omega (h_{2}g_{2},k_{2},l_{2}) &=&\omega
(g_{1},k_{1},l_{1})\omega (h_{1},g_{2}k_{2},l_{2})\omega (h_{2},g_{3},k_{3})
\label{cocycle omega} \\
\omega (h,1_{H},g) &=&\varepsilon (h)\varepsilon (g)
\end{eqnarray}%
hold for all $h,g,k,l\in H$.
\end{definition}

As a consequence, we have also $\omega (1_{H},h,g)=\omega
(h,g,1_{H})=\varepsilon (h)\varepsilon (g)$ for each $g,h\in H$.

\begin{definition}
A coquasi-Hopf algebra is a coquasi-bialgebra $H$ endowed with a coalgebra
antihomomorphism $S:H\longrightarrow H$ (the antipode) and with elements $%
\alpha $, $\beta \in H^{\ast }$ satisfying 
\begin{eqnarray}
S(h_{1})\alpha (h_{2})h_{3} &=&\alpha (h)1_{H}  \label{SalfaId} \\
h_{1}\beta (h_{2})S(h_{3}) &=&\beta (h)1_{H}  \label{IdbetaS} \\
\omega (h_{1}\beta (h_{2}),S(h_{3}),\alpha (h_{4})h_{5}) &=&\omega
^{-1}(S(h_{1}),\alpha (h_{2})h_{3}\beta (h_{4}),S(h_{5}))=\varepsilon (h)
\label{omega anihileaza S}
\end{eqnarray}%
for all $h\in H$.
\end{definition}

These relations imply also $S(1_{H})=1_{H}$ and $\alpha (1_{H})\beta
(1_{H})=1$, so by rescaling $\alpha $ and $\beta $, we may assume that $%
\alpha (1_{H})=1$ and $\beta (1_{H})=1$. The antipode is unique up to a
convolution invertible element $U\in H^{\ast }$: if $(S^{\prime },\alpha
^{\prime },\beta ^{\prime })$ is another triple with the above properties,
then according to \cite{Majid95} we have 
\begin{equation}
S^{\prime }(h)=U(h_{1})S(h_{2})U^{-1}(h_{3}),\quad \alpha ^{\prime
}(h)=U(h_{1})\alpha (h_{2}),\quad \beta ^{\prime }(h)=\beta
(h_{1})U^{-1}(h_{2})  \label{change antipode coquasi}
\end{equation}%
for all $h\in H$.

We shall use in this paper the monoidal structure of the category of right
(left) $H$-comodules and of the category of $H$-bicomodules: the tensor
product is over the base field and the comodule structure (left or right) of
the tensor product is the codiagonal one. The reassociators are 
\begin{eqnarray*}
\Phi _{U,V,W} &:&(U\otimes V)\otimes W\longrightarrow U\otimes (V\otimes W)
\\
\Phi _{U,V,W}((u\otimes v)\otimes w) &=&u_{0}\otimes (v_{0}\otimes
w_{0})\omega (u_{1},v_{1},w_{1})
\end{eqnarray*}%
for $u\in U$, $v\in V$, $w\in W$ and $U,V,W\in \mathcal{M}^{H}$,
respectively 
\begin{eqnarray*}
\Phi _{U,V,W} &:&(U\otimes V)\otimes W\longrightarrow U\otimes (V\otimes W)
\\
\Phi _{U,V,W}((u\otimes v)\otimes w) &=&\omega
^{-1}(u_{-1},v_{-1},w_{-1})u_{0}\otimes (v_{0}\otimes w_{0})
\end{eqnarray*}%
for $u\in U$, $v\in V$, $w\in W$ and $U,V,W\in {}^{H}\mathcal{M}$. For the
category of $H$-bicomodules, one can obtain the reassociator by combining
the above two, namely by multiplication to the left by $\omega ^{-1}$,
respectively to the right by $\omega $.

For $H$ a coquasi-bialgebra, the linear dual $H^{\ast }=Hom(H,\Bbbk )$
becomes an associative algebra with multiplication given by the usual
convolution product%
\begin{equation}
(h^{\ast }g^{\ast })(h)=h^{\ast }(h_{1})g^{\ast }(h_{2})\qquad \forall h\in
H\quad \text{and}\quad h^{\ast },g^{\ast }\in H^{\ast }
\label{convolutie pe hrond}
\end{equation}%
and unit $\varepsilon $. This algebra is acting on $H$ by the formulas:%
\begin{equation}
h^{\ast }\rightharpoonup h=h_{1}h^{\ast }(h_{2})\text{,}\qquad
h\leftharpoonup h^{\ast }=h^{\ast }(h_{1})h_{2}
\label{actiunea lui Hrond* pe Hrond}
\end{equation}%
for any $h^{\ast }\in H^{\ast }$, $h\in H$.

Even though $H$ is not an associative algebra, we keep the notation from the
Hopf algebra case for the weak action of $H$ on $H^{\ast }$ 
\begin{equation}
(h\rightharpoonup h^{\ast })(g)=h^{\ast }(gh)\text{, }\qquad (h^{\ast
}\leftharpoonup h)(g)=h^{\ast }(hg)
\label{actiunea slaba a lui Hrond pe Hrond*}
\end{equation}%
for any $h^{\ast }\in H^{\ast }$, $g,h\in H$.

If $H$ is a finite dimensional coquasi-bialgebra, then it is easy to check
that $H^{\ast }$ is a quasi-bialgebra with the induced dual operations and
conversely, the linear dual of any finite dimensional quasi-bialgebra
becomes a coquasi-bialgebra, which justifies some common notations and
definitions. An immediate consequence is the identification between the
category of right $H$-comodules $\mathcal{M}^{H}$ and the category of left $%
H^{\ast }$-modules $_{H^{\ast }}\mathcal{M}$. A right $H$-comodule $V$
becomes a left $H^{\ast }$-module by $h^{\ast }v=h^{\ast }(v_{1})v_{0}$, $%
\forall $ $h^{\ast }\in H^{\ast }$, $v\in V$. Conversely, to any left $%
H^{\ast }$-module $V$ we may associate an $H$-coaction by $\rho
_{V}(v)=\sum\limits_{i=1}^{\dim H}e^{i}v\otimes e_{i}$, where again $%
(e_{i})_{i=1,\dim H}$ and $(e^{i})_{i=1,\dim H}$ are dual bases for $H$,
respectively $H^{\ast }$.

Now, recall from \cite{Panaite97Stefan} the following: for $\tau \in
(H\otimes H)^{\ast }$ a convolution invertible map such that $\tau
(1,h)=\tau (h,1)=\varepsilon (h)$ for all $h\in H$ ($\tau $ is called a
twist or a gauge transformation), one can define a new structure of
coquasi-bialgebra (or coquasi-Hopf algebra) on $H$, denoted $H_{\tau }$, by
taking 
\begin{eqnarray}
h\cdot _{\tau }g &=&\tau (h_{1},g_{1})h_{2}g_{2}\tau ^{-1}(h_{3},g_{3})
\label{multiplic Htau} \\
\omega _{\tau }(h,g,k) &=&\tau (g_{1},k_{1})\tau (h_{1},g_{2}k_{2})\omega
(h_{2},g_{3},k_{3})\tau ^{-1}(h_{3}g_{4},k_{4})\tau ^{-1}(h_{4},g_{5})
\label{asociator Htau} \\
\alpha _{\tau }(h) &=&\tau ^{-1}(S(h_{1}),\alpha (h_{2})h_{3})
\label{alfa Htau} \\
\beta _{\tau }(h) &=&\tau (h_{1}\beta (h_{2}),S(h_{3}))  \label{beta Htau}
\end{eqnarray}%
for all $h,g,k\in H$, and keeping the unit, the comultiplication, the counit
and the antipode unchanged.

\begin{remark}
\label{monoidal isomorfism gauge}There is a monoidal isomorphism $\mathcal{M}%
^{H}\simeq \mathcal{M}^{H_{\tau }}$, which is the identity on objects and on
morphisms, with monoidal structure given by $V\otimes W\longrightarrow
V\otimes W$, $v\otimes w\longrightarrow v_{0}\otimes w_{0}\tau (v_{1},w_{1})$%
, where $v\in V$, $w\in W$ and $V$, $W\in \mathcal{M}^{H}$.
\end{remark}

We shall also need a particular twist $\mathbf{f}\in (H\otimes H)^{\ast }$,
which appears in \cite{Bulacu99} and controls how far is the antipode $S$
from a anti-algebra morphism:%
\begin{equation}
\mathbf{f}(h_{1},g_{1})S(h_{2}g_{2})=S(g_{1})S(h_{1})\mathbf{f}%
(h_{2},g_{2})\qquad \text{for all }h,g\in H  \label{twist f}
\end{equation}%
We have also from \cite{Bulacu02co} that%
\begin{equation}
\beta (h_{1}g_{1})\mathbf{f}^{(-1)}(h_{2},g_{2})=\omega
(h_{1}g_{1},S(g_{5}),S(h_{4}))\omega ^{-1}(h_{2},g_{2},S(g_{4}))\beta
(h_{3})\beta (g_{3})  \label{1beta*f-1=delta}
\end{equation}

\begin{definition}
(\cite{Bulacu00}) A right comodule algebra $A$ over a coquasi-bialgebra $H$
is an algebra in the monoidal category $\mathcal{M}^{H}$. This means $%
(A,\rho _{A})$ is a right $H$-comodule with a multiplication map $\mu
_{A}:A\otimes A\longrightarrow A$, denoted $\mu _{A}(a\otimes b)=ab$, for $%
a,b\in A$, and a unit map $u_{A}:\Bbbk \longrightarrow A$, where we put $%
u_{A}(1)=1_{A}$, which are both $H$-colinear, such that 
\begin{equation}
(ab)c=a_{0}(b_{0}c_{0})\omega (a_{1},b_{1},c_{1})  \label{asoc comod alg}
\end{equation}%
holds for any $a,b,c\in A$.
\end{definition}

\begin{definition}
(\cite{Bulacu00}) For $H$ a coquasi-bialgebra and $A$ a right $H$-comodule
algebra, we may define the notion of right module over $A$ in the category $%
\mathcal{M}^{H}$.\ Explicitly, this is a right $H$-comodule $(M,\rho _{M})$,
endowed with a right $A$-action $\mu _{M}:M\otimes A\longrightarrow M$,
denoted $\mu _{M}(m,a)=ma$, such that 
\begin{eqnarray*}
(ma)b &=&m_{0}(a_{0}b_{0})\omega (m_{1},a_{1},b_{1}) \\
m1_{A} &=&m \\
\rho _{M}(ma) &=&m_{0}a_{0}\otimes m_{1}a_{1}
\end{eqnarray*}%
hold for all $m\in M$, $a,b\in A$. The category of such objects, with
morphisms the right $H$-colinear maps which respect the $A$-action, is
called the category of relative right\textbf{\ }$(H,A)$\textbf{-}Hopf
modules and denoted $\mathcal{M}_{A}^{H}$.
\end{definition}

\begin{remark}
\label{twist comodule algebra}It was proven in \cite{Bulacu00} that if $\tau 
$ is a twist on $H$, then the formula 
\begin{equation}
a\cdot _{\tau }b=a_{0}b_{0}\tau ^{-1}(a_{1},b_{1})  \label{multiplic Atau}
\end{equation}%
for all $a,b\in A$ defines a new multiplication such that $A$, with this new
multiplication (denoted $A_{\tau ^{-1}}$) becomes a right $H_{\tau }$%
-comodule algebra.\ It is easy to see that the isomorphism of Remark \ref%
{monoidal isomorfism gauge} sends the algebra $A$ of the monoidal category $%
\mathcal{M}^{H}$ exactly to the algebra $A_{\tau ^{-1}}$ in $\mathcal{M}%
^{H_{\tau }}$. Therefore the categories of right relative Hopf modules $%
\mathcal{M}_{A}^{H}$ and $\mathcal{M}_{A_{\tau ^{-1}}}^{H_{\tau }}$ will
also be isomorphic.
\end{remark}

Let $A$ be a right $H$-comodule algebra. Consider the space of coinvariants%
\begin{equation*}
B=A^{coH}=\{a\in A\left\vert \rho _{A}(a)=a\otimes 1_{H}\right. \}
\end{equation*}%
It is immediate that this is an associative $\Bbbk $-algebra with unit and
multiplication induced by the unit and the multiplication of $A$. There is a
pair of adjoint functors which arises naturally, namely the induced and the
coinvariant functors 
\begin{equation}
\mathcal{M}_{B}\overset{(-)\otimes _{B}A}{\underset{(-)^{coH}}{\mathcal{%
\rightleftarrows }}}\mathcal{M}_{A}^{H}  \label{adjunctie}
\end{equation}%
where $N\otimes _{B}A$ is a relative Hopf module with action and coaction
induced by $A$, and $M^{coH}$ becomes naturally a right $B$-module by
restricting the scalars, for any $N\in \mathcal{M}_{B}$ and $M\in \mathcal{M}%
_{A}^{H}$. Notice also the natural isomorphism 
\begin{equation}
Hom_{A}^{H}(A,M)\simeq M^{coH}  \label{coH=hom}
\end{equation}%
for any $M\in \mathcal{M}_{A}^{H}$. Finally, we recall from \cite{Balan07}
the notion of a Galois extension:

\begin{definition}
(\cite{Balan07})\label{def Galois} Let $H$ a coquasi-Hopf algebra and $A$ a
right $H$-comodule algebra with coinvariants $B=A^{coH}$. The extension $%
B\subseteq A$ is Galois if the map $can:A\otimes _{B}A\longrightarrow
A\otimes H$, given by 
\begin{equation}
a\otimes _{B}b\longrightarrow a_{0}b_{0}\otimes b_{4}\omega
^{-1}(a_{1},b_{1}\beta (b_{2}),S(b_{3}))  \label{can}
\end{equation}%
is bijective.
\end{definition}

Although this definition implies the existence of the antipode, unlike the
classical associative case, it is deeply connected with the above mentioned
adjunction of functors, exactly as for Hopf algebras (see \cite{Balan07}).

\section{\label{crossed product}Crossed products by coquasi-bialgebras}

\subsection{Definition of a crossed product}

We start by developing a suitable theory of crossed products, generalizing
that of \cite{Blattner86} and \cite{Doi86}. Let $H$ be a coquasi-bialgebra
and $R$ an associative algebra. On $R$ we consider the following structures:

\begin{itemize}
\item a weak action $\cdot :H\otimes R\longrightarrow R$, meaning a bilinear
map such that%
\begin{equation}
h\cdot (rs)=(h_{1}\cdot r)(h_{2}\cdot s),\quad h\cdot 1_{R}=\varepsilon
(h)1_{R}  \label{weak action}
\end{equation}%
for all $h\in H$ and $r,s\in R$;

\item a linear map $\sigma :H\otimes H\longrightarrow R$.
\end{itemize}

\begin{definition}
The crossed product algebra $R\overline{\#}_{\sigma }H$ is $R\otimes H$ as
vector space with multiplication 
\begin{equation}
(r\overline{\#}_{\sigma }h)(s\overline{\#}_{\sigma }g)=r(h_{1}\cdot s)\sigma
(h_{2},g_{1})\overline{\#}_{\sigma }h_{3}g_{2}
\label{multiplicationcrossedproduct}
\end{equation}
\end{definition}

And the following Theorem explains what is this new structure:

\begin{theorem}
\label{conditii produs incrucisat}$R\overline{\#}_{\sigma }H$ is a right $H$%
-comodule algebra, with unit $1_{R}\overline{\#}_{\sigma }1_{H}$ and
coaction $I_{R}\otimes \Delta $ if and only if the following relations are
satisfied:%
\begin{eqnarray}
1_{H}\cdot r &=&r  \label{unit} \\
\lbrack h_{1}\cdot (g_{1}\cdot r)]\sigma (h_{2},g_{2}) &=&\sigma
(h_{1},g_{1})[(h_{2}g_{2})\cdot r]  \label{asociat} \\
\sigma (h,1) &=&\sigma (1,h)=\varepsilon (h)1_{R}  \label{vanish cocycle} \\
\lbrack h_{1}\cdot \sigma (g_{1},l_{1})]\sigma (h_{2},g_{2}l_{2}) &=&\sigma
(h_{1},g_{1})\sigma (h_{2}g_{2},l_{1})\omega ^{-1}(h_{3},g_{3},l_{2})
\label{cocycle}
\end{eqnarray}%
for all $r\in R$, $h,g,l\in H$. In this case we say that $(R,\cdot ,\sigma )$
form an $H$-crossed system and that $\sigma $ is a $2$-cocycle.
\end{theorem}

\begin{proof}
It is obvious that $R\overline{\#}_{\sigma }H$ becomes an $H$-comodule via $%
I_{R}\otimes \Delta $. Let's check the colinearity of the multiplication: 
\begin{eqnarray*}
\rho _{R\overline{\#}_{\sigma }H}((r\overline{\#}_{\sigma }h)(s\overline{\#}%
_{\sigma }g)) &=&\rho _{R\overline{\#}_{\sigma }H}(r(h_{1}\cdot s)\sigma
(h_{2},g_{1})\overline{\#}_{\sigma }h_{3}g_{2}) \\
&=&r(h_{1}\cdot s)\sigma (h_{2},g_{1})\overline{\#}_{\sigma
}h_{3}g_{2}\otimes h_{4}g_{3} \\
&=&(r\overline{\#}_{\sigma }h_{1})(s\overline{\#}_{\sigma }g_{1})\otimes
h_{2}g_{2}
\end{eqnarray*}%
for any $r\overline{\#}_{\sigma }h,s\overline{\#}_{\sigma }g\in R\overline{\#%
}_{\sigma }H$. Next, if the above conditions are fulfilled, then 
\begin{eqnarray*}
\lbrack (r\overline{\#}_{\sigma }h)(s\overline{\#}_{\sigma }g)](t\overline{\#%
}_{\sigma }k) &=&[r(h_{1}\cdot s)\sigma (h_{2},g_{1})\overline{\#}_{\sigma
}h_{3}g_{2}](t\overline{\#}_{\sigma }k) \\
(\ref{multiplicationcrossedproduct}) &=&r(h_{1}\cdot s)\sigma
(h_{2},g_{1})[(h_{3}g_{2})\cdot t]\sigma (h_{4}g_{3},k_{1})\overline{\#}%
_{\sigma }(h_{5}g_{4})k_{2} \\
&=&r(h_{1}\cdot s)\sigma (h_{2},g_{1})[(h_{3}g_{2})\cdot t]\sigma
(h_{4}g_{3},k_{1})\overline{\#}_{\sigma }\omega
^{-1}(h_{5},g_{4},k_{2})h_{6}(g_{5}k_{3})\omega (h_{7},g_{6},k_{4}) \\
(\ref{asociat}) &=&r(h_{1}\cdot s)[h_{2}\cdot (g_{1}\cdot t)]\sigma
(h_{3},g_{2})\sigma (h_{4}g_{3},k_{1})\overline{\#}_{\sigma }\omega
^{-1}(h_{5},g_{4},k_{2})h_{6}(g_{5}k_{3})\omega (h_{7},g_{6},k_{4}) \\
(\ref{cocycle}) &=&r(h_{1}\cdot s)[h_{2}\cdot (g_{1}\cdot t)][h_{3}\cdot
\sigma (g_{2},k_{1})]\sigma (h_{4},g_{3}k_{2})\overline{\#}_{\sigma
}h_{5}(g_{4}k_{3})\omega (h_{6},g_{5},k_{4}) \\
&=&r\{h_{1}\cdot \lbrack s(g_{1}\cdot t)\sigma (g_{2},k_{1})]\}\sigma
(h_{2},g_{3}k_{2})\overline{\#}_{\sigma }h_{3}(g_{4}k_{3})\omega
(h_{4},g_{5},k_{4}) \\
(\ref{multiplicationcrossedproduct}) &=&(r\overline{\#}_{\sigma
}h_{1})[s(g_{1}\cdot t)\sigma (g_{2},k_{1})\overline{\#}_{\sigma
}g_{3}k_{2}]\omega (h_{2},g_{4},k_{3}) \\
(\ref{multiplicationcrossedproduct}) &=&(r\overline{\#}_{\sigma }h_{1})[(s%
\overline{\#}_{\sigma }g_{1})(t\overline{\#}_{\sigma }k_{1})]\omega
(h_{2},g_{2},k_{2})
\end{eqnarray*}%
for all $r\overline{\#}_{\sigma }h,s\overline{\#}_{\sigma }g,t\overline{\#}%
_{\sigma }k\in R\overline{\#}_{\sigma }H$. Finally, we have 
\begin{eqnarray*}
(1_{R}\overline{\#}_{\sigma }1_{H})(r\overline{\#}_{\sigma }h)
&=&1_{R}(1_{H}\cdot r)\sigma (1_{H},h_{1})\overline{\#}_{\sigma }1_{H}h_{2}
\\
(\ref{unit}),(\ref{vanish cocycle}) &=&r\overline{\#}_{\sigma }h
\end{eqnarray*}%
and%
\begin{eqnarray*}
(r\overline{\#}_{\sigma }h)(1_{R}\overline{\#}_{\sigma }1_{H})
&=&r(h_{1}\cdot 1_{R})\sigma (h_{2},1_{H})\overline{\#}_{\sigma }h_{3}1_{H}
\\
(\ref{vanish cocycle}),(\ref{weak action}) &=&r\overline{\#}_{\sigma }h
\end{eqnarray*}%
so we obtain an algebra in the monoidal category of right $H$-comodules.
Conversely, $(1_{R}\overline{\#}_{\sigma }1_{H})(r\overline{\#}_{\sigma
}1_{H})=r\overline{\#}_{\sigma }1_{H}$ gives $1_{H}\cdot r=r$. Also, $(1_{R}%
\overline{\#}_{\sigma }1_{H})(1_{R}\overline{\#}_{\sigma }h)=1_{R}\overline{%
\#}_{\sigma }h$ implies $\sigma (1_{H},h_{1})\overline{\#}_{\sigma
}h_{2}=1_{R}\overline{\#}_{\sigma }h$ and applying $\varepsilon $ on the
second component gives us $\sigma (1_{H},h)=\varepsilon (h)$. Similarly, $%
(1_{R}\overline{\#}_{\sigma }h)(1_{R}\overline{\#}_{\sigma }1_{H})=1_{R}%
\overline{\#}_{\sigma }h$ implies $\sigma (h,1_{H})=\varepsilon (h)$. For
the last identity, write down the associativity of the crossed product
algebra in the monoidal category of comodules, and compute the product $%
[(1_{R}\overline{\#}_{\sigma }h)(1_{R}\overline{\#}_{\sigma }g)](1_{R}%
\overline{\#}_{\sigma }l)$ in the two possible ways. At the end, apply $%
\varepsilon $ on the second tensorand. Finally, in order to obtain relation (%
\ref{asociat}), repeat this procedure for the product $[(1_{R}\overline{\#}%
_{\sigma }h)(1_{R}\overline{\#}_{\sigma }g)](r\overline{\#}_{\sigma }1_{H})$.
\end{proof}

\begin{remark}
\begin{enumerate}
\item \label{cazuri particulare pt crossed}If $H$ is a bialgebra, then we
recover the usual definition of the crossed product of an algebra by the
bialgebra $H$. Therefore all known examples for the associative case fit in
our picture.

\item If the cocycle $\sigma $ is trivial (i.e. $\sigma (h,g)=\varepsilon
(h)\varepsilon (g)1_{R}$, for $h,g\in H$), then by relation (\ref{cocycle})\
it follows that $\omega $ is also trivial. Hence $H$ is a bialgebra and $R$
is a left $H$-module algebra. The result is the usual smash product $R\#H$.

\item For the trivial weak action (i.e. $h\cdot r=\varepsilon (h)r$, where $%
h\in H$, $r\in R$), relation (\ref{asociat}) implies $\func{Im}\sigma
\subseteq Z(R)$ and by (\ref{cocycle}) we have 
\begin{equation}
\sigma (g_{1},l_{1})\sigma (h,g_{2}l_{2})=\sigma (h_{1},g_{1})\sigma
(h_{2}g_{2},l_{1})\omega ^{-1}(h_{3},g_{3},l_{2})  \label{sigma omega}
\end{equation}%
for all $h,g,l\in H$.

\item If $R=\Bbbk $, then the weak action must be trivial, according to (\ref%
{weak action}) and $\sigma $ is a twist on $H$. Hence by (\ref{sigma omega}%
)\ it follows that $H$ is a deformation of a bialgebra by the twist $\sigma $%
. Therefore there are no crossed products of the base field by a nontrivial
coquasi-bialgebra.
\end{enumerate}
\end{remark}

Before ending this Section, we shall notice the following relation, which
will be used later on:

\begin{proposition}
If $\sigma $ is convolution invertible and the relations (\ref{weak action}%
), (\ref{unit}), (\ref{asociat}), (\ref{vanish cocycle}), (\ref{cocycle})
are satisfied, then%
\begin{equation}
h\cdot \sigma ^{-1}(g,l)=\sigma (h_{1},g_{1}l_{1})\omega
(h_{2},g_{2},l_{2})\sigma ^{-1}(h_{3}g_{3},l_{3})\sigma ^{-1}(h_{4},g_{4})
\label{h ori sigma la -1}
\end{equation}%
for all $h,g,l\in H$.
\end{proposition}

\begin{proof}
By (\ref{weak action}), it follows that the map $H\otimes H\otimes
H\longrightarrow R$, $(h,g,l)\longrightarrow h\cdot \sigma (g,l)$ is
convolution invertible, with convolution inverse $(h,g,l)\longrightarrow
h\cdot \sigma ^{-1}(g,l)$. As $\sigma $ is invertible, relation (\ref%
{cocycle}) implies 
\begin{equation*}
h\cdot \sigma (g,l)=\sigma (h_{1},g_{1})\sigma (h_{2}g_{2},l_{1})\omega
^{-1}(h_{3},g_{3},l_{2})\sigma ^{-1}(h_{4},g_{4}l_{3})
\end{equation*}%
But the map $(h,g,l)\longrightarrow \sigma (h_{1},g_{1})\sigma
(h_{2}g_{2},l_{1})\omega ^{-1}(h_{3},g_{3},l_{2})\sigma
^{-1}(h_{4},g_{4}l_{3})$ is easily checked to be convolution invertible,
with inverse $(h,g,l)\longrightarrow \sigma (h_{1},g_{1}l_{1})\omega
(h_{2},g_{2},l_{2})\sigma ^{-1}(h_{3}g_{3},l_{3})\sigma ^{-1}(h_{4},g_{4})$.
By the uniqueness of the inverse of an element in the convolution algebra $%
Hom(H\otimes H\otimes H,R)$ we get the desired formula.
\end{proof}

\subsection{Examples of crossed products}

We shall now provide some examples of crossed products by coquasi-bialgebras.

\begin{enumerate}
\item \label{trivial crossed product}By Remark 10.(4), if we start with a
bialgebra $H$ and a twist $\tau :H\otimes H\longrightarrow \Bbbk $, then we
may form the crossed product $\Bbbk \#_{\tau ^{-1}}H_{\tau }$ of the base
field by the deformed coquasi-bialgebra $H_{\tau }$, having as cocycle the
convolution inverse of the twist. We apply now this construction to the
group algebra of a finite group $G$, in particular to $C_{2}^{n}$. It
follows from \cite{Albuquerque99}\ that all Cayley and Clifford algebras can
be obtained in this way, as crossed products by some coquasi-bialgebras.

\item Take $H$ a coquasi-bialgebra. Then $H^{\ast }$ is an associative
algebra with multiplication given by (\ref{convolutie pe hrond}). The
formula (\ref{actiunea slaba a lui Hrond pe Hrond*}) defines a weak action
of $H$ on $H^{\ast }$. It is easy to check that (\ref{weak action}) holds.
Define now $\sigma :H\otimes H\longrightarrow H^{\ast }$, by $\sigma
(h,g)(k)=\omega ^{-1}(k,h,g)$. Then all the relations (\ref{weak action}), (%
\ref{unit}), (\ref{asociat}), (\ref{vanish cocycle}), (\ref{cocycle}) hold.
We obtain thus the crossed product $H^{\ast }\overline{\#}_{\sigma }H$,
which in Hopf algebra case reduces to the Heisenberg double.

\item The previous example can be generalized as follows: let $H$ be a
coquasi-bialgebra and $A$ a right $H$-comodule algebra. On the vector space $%
Hom(H,A)$ we define the following multiplication:%
\begin{equation*}
(\varphi \circledast \psi )(h)=\varphi (\psi (h_{3})_{2}h_{2})_{0}\psi
(h_{3})_{0}\omega ^{-1}(\varphi (\psi (h_{3})_{2}h_{2})_{1},\psi
(h_{3})_{1},h_{1})
\end{equation*}%
for all $\varphi ,\psi \in Hom(H,A)$ and $h\in H$. Then $(Hom(H,A),%
\circledast )$ becomes an associative algebra with unit $\varepsilon 1_{A}$:%
\begin{allowdisplaybreaks}%
\begin{eqnarray*}
((\varphi \circledast \psi )\circledast \lambda )(h) &=&(\varphi \circledast
\psi )(\lambda (h_{3})_{2}h_{2})_{0}\lambda (h_{3})_{0}\omega ^{-1}((\varphi
\circledast \psi )(\lambda (h_{3})_{2}h_{2})_{1},\lambda (h_{3})_{1},h_{1})
\\
&=&[\varphi (\psi (\lambda (h_{5})_{4}h_{4})_{3}(\lambda
(h_{5})_{3}h_{3}))_{0}\psi (\lambda (h_{5})_{4}h_{4})_{0}]\lambda (h_{5})_{0}
\\
&&\omega ^{-1}(\varphi (\psi (\lambda (h_{5})_{4}h_{4})_{3}(\lambda
(h_{5})_{3}h_{3}))_{2},\psi (\lambda (h_{5})_{4}h_{4})_{2},\lambda
(h_{5})_{2}h_{2}) \\
&&\omega ^{-1}(\varphi (\psi (\lambda (h_{5})_{4}h_{4})_{3}(\lambda
(h_{5})_{3}h_{3}))_{1}\psi (\lambda (h_{5})_{4}h_{4})_{1},\lambda
(h_{5})_{1},h_{1}) \\
(\ref{cocycle omega}) &=&[\varphi (\psi (\lambda
(h_{5})_{4}h_{4})_{4}(\lambda (h_{5})_{3}h_{3}))_{0}\psi (\lambda
(h_{5})_{4}h_{4})_{0}]\lambda (h_{5})_{0} \\
&&\omega ^{-1}(\varphi (\psi (\lambda (h_{5})_{4}h_{4})_{4}(\lambda
(h_{5})_{3}h_{3}))_{1},\psi (\lambda (h_{5})_{4}h_{4})_{1},\lambda
(h_{5})_{1}) \\
&&\omega ^{-1}(\varphi (\psi (\lambda (h_{5})_{4}h_{4})_{4}(\lambda
(h_{5})_{3}h_{3}))_{2},\psi (\lambda (h_{5})_{4}h_{4})_{2}\lambda
(h_{5})_{2},h_{1}) \\
&&\omega ^{-1}(\psi (\lambda (h_{5})_{4}h_{4})_{3},\lambda (h_{5})_{3},h_{2})
\\
(\ref{asociat multipl}) &=&\varphi (\psi (\lambda
(h_{5})_{4}h_{4})_{3}(\lambda (h_{5})_{3}h_{3}))_{0}[\psi (\lambda
(h_{5})_{4}h_{4})_{0}\lambda (h_{5})_{0}] \\
&&\omega ^{-1}(\varphi (\psi (\lambda (h_{5})_{4}h_{4})_{3}(\lambda
(h_{5})_{3}h_{3}))_{1},\psi (\lambda (h_{5})_{4}h_{4})_{1}\lambda
(h_{5})_{1},h_{1}) \\
&&\omega ^{-1}(\psi (\lambda (h_{5})_{4}h_{4})_{2},\lambda (h_{5})_{2},h_{2})
\\
&=&\varphi ((\psi \circledast \lambda )(h_{3})_{2}h_{2})_{0}(\psi
\circledast \lambda )(h_{3})_{0} \\
&&\omega ^{-1}(\varphi ((\psi \circledast \lambda
)(h_{3})_{2}h_{2})_{1},(\psi \circledast \lambda )(h_{3})_{1},h_{1}) \\
&=&(\varphi \circledast (\psi \circledast \lambda ))(h)
\end{eqnarray*}%
\end{allowdisplaybreaks}%
This algebra may be seen as a generalization of Doi's smash product $\#(H,A)$
(\cite{Doi94}), where $H$ is a Hopf algebra and $A$ a comodule algebra.
Consider now the maps%
\begin{eqnarray*}
(h\cdot \varphi )(g) &=&\varphi (gh) \\
\sigma (h,g)(k) &=&\omega ^{-1}(k,h,g)1_{A}
\end{eqnarray*}%
Then we may form the crossed product $Hom(H,A)\overline{\#}_{\sigma }H$. The
particular case $A=\Bbbk $ reduces to the previous example. If $H$ is a
co-Frobenius Hopf algebra, restricting the crossed product to the subalgebra
(with local units) $A\#H^{\ast rat}\subseteq Hom(H,A)$, gives the
isomorphism $(A\#H^{\ast rat})\#H\simeq M_{H}^{f}(A)$, where $M_{H}^{f}(A)$
is the ring of matrices with rows and columns indexed by a basis of $H$,
with only finitely many non-zero entries in $A$ (\cite{Dascalescu00},
Theorem 6.5.11). It remains an open question if such a duality result holds
also for co-Frobenius coquasi-Hopf algebras.

\item Let $\mathfrak{H}$ be a finite dimensional quasi-bialgebra and $%
(R,\rho ,\phi _{\rho })$ a right $\mathfrak{H}$-comodule algebra (that is,
an associative algebra endowed with an algebra morphism $\rho
:R\longrightarrow R\otimes \mathfrak{H}$ and an invertible element $\phi
_{\rho }\in R\otimes \mathfrak{H}\otimes \mathfrak{H}$, satisfying some
compatibility conditions), as it was defined in \cite{Bulacu02} and \cite%
{Hausser99}. Denote $\phi _{\rho }^{-1}=x_{\rho }^{1}\otimes x_{\rho
}^{2}\otimes x_{\rho }^{3}$, summation understood. Then $H=\mathfrak{H}%
^{\ast }$ is a coquasi-bialgebra, and if we put $\sigma (\varphi ,\psi
)=x_{\rho }^{1}\varphi (x_{\rho }^{2})\psi (x_{\rho }^{3})$, for $\varphi $, 
$\psi \in H$, we obtain a cocycle as before. Taking also the weak action
given by $\varphi \cdot r=r_{0}\varphi (r_{1})$, where we have denoted $\rho
(r)=r_{0}\otimes r_{1}$, $r\in R$, we get the trivial example of a crossed
product algebra $R\overline{\#}_{\sigma }\mathfrak{H}^{\ast }=R\overline{\#}%
\mathfrak{H}^{\ast }$, namely the quasi-smash product, as it was defined in 
\cite{Bulacu02}. Conversely, if $H$ is a coquasi-bialgebra and $R$ an
associative algebra endowed with a weak action $\cdot $\ and a convolution
invertible $2$-cocycle $\sigma $, which fulfill relations (\ref{weak action}%
), (\ref{unit}), (\ref{asociat}), (\ref{cocycle}), (\ref{vanish cocycle}),
then $H^{\ast }$ is a quasi-bialgebra which coact weakly on $R$ by $\rho
(r)=\sum\limits_{i=1}^{\dim H}e_{i}\cdot r\otimes e^{i}$, where $%
(e_{i})_{i=1,\dim H}$ and $(e^{i})_{i=1,\dim H}$ are dual bases for $H,$
respectively for $H^{\ast }$. If we denote by $\phi _{\rho }\in R\otimes
H^{\ast }\otimes H^{\ast }$ the invertible element $\phi _{\rho }(h\otimes
g)=\sigma ^{-1}(h,g)$ (here we have used the vector space isomorphism $%
R\otimes H^{\ast }\otimes H^{\ast }\simeq Hom(H\otimes H,R)$), we obtain
that $R$ is a right $H^{\ast }$-comodule algebra.
\end{enumerate}

\subsection{Crossed products viewed towards monoidal categories}

Let $H$ be a coquasi-bialgebra and $R$ an associative algebra, endowed with
a weak action $\cdot $\ and a linear map $\sigma $, as in the beginning of
the previous Section. Remember that the category of left $H$-comodules $^{H}%
\mathcal{M}$ is monoidal. Now, for each right $R$-module $M\in \mathcal{M}%
_{R}$ and left $H$-comodule $V\in {}^{H}\mathcal{M}$, we define on $M\otimes
V$ the following structure:%
\begin{equation*}
(m\otimes v)r=m(v_{-1}\cdot r)\otimes v_{0}
\end{equation*}%
for any $m\in M$, $v\in V$, $r\in R$.

\begin{proposition}
\begin{enumerate}
\item With the previous notations, $M\otimes V$ is a right $R$-module if and
only if the conditions (\ref{weak action}) are fulfilled.

\item For $M\in \mathcal{M}_{R}$ and $V,W\in {}^{H}\mathcal{M}$, consider
the map 
\begin{eqnarray*}
\Psi _{M,V,W} &:&(M\otimes V)\otimes W\longrightarrow M\otimes (V\otimes W)
\\
(m\otimes v)\otimes w &\longrightarrow &m\sigma (v_{-1},w_{-1})\otimes
(v_{0}\otimes w_{0})
\end{eqnarray*}%
Then $\Psi _{M,V,W}$ is right $R$-linear if and only (\ref{asociat}) holds.

\item $\mathcal{M}_{R}$ becomes a right $^{H}\mathcal{M}$-category with the
above structures if and only if (\ref{weak action}), (\ref{unit}), (\ref%
{asociat}), (\ref{cocycle}), (\ref{vanish cocycle}) hold.
\end{enumerate}
\end{proposition}

\begin{proof}
Straightforward. Verifications are left to the reader.
\end{proof}

Hence we have obtained a categorical explanation for the conditions imposed
on the weak action and on the cocycle $\sigma $.

We make the observation that a right action of $\mathcal{M}^{H}$ on $_{R}%
\mathcal{M}$ can be constructed similarly. We can go even further, by
considering bicomodules instead of one-sided comodules. All we need is to
change properly the category on which the action is considered. We start by
noticing that $R$, with trivial coaction, can be viewed as an algebra in the
monoidal category of right $H$-comodules. Therefore we may consider the
category of right $R$-modules $\mathcal{M}_{R}^{H}$ within this monoidal
category (an object $M$ is a right $R$-module, right $H$-comodule with
comodule map $\rho _{M}(m)=m_{0}\otimes m_{1}$ such that $\rho
_{M}(mr)=m_{0}r\otimes m_{1}$, $\forall m\in M$, $r\in R$). Now, for each $%
M\in \mathcal{M}_{R}^{H}$ and $V\in {}^{H}\mathcal{M}^{H}$, we define on $%
M\otimes V$ the following structures:%
\begin{eqnarray*}
(m\otimes v)r &=&m(v_{-1}\cdot r)\otimes v_{0} \\
\rho (m\otimes v) &=&m_{0}\otimes v_{0}\otimes m_{1}v_{1}
\end{eqnarray*}%
for any $m\in M$, $v\in V$, $r\in R$. Also, for $M\in \mathcal{M}_{R}^{H}$
and $V,W\in {}^{H}\mathcal{M}^{H}$, consider the map 
\begin{eqnarray*}
\Psi _{M,V,W} &:&(M\otimes V)\otimes W\longrightarrow M\otimes (V\otimes W)
\\
(m\otimes v)\otimes w &\longrightarrow &m\sigma (v_{-1},w_{-1})\otimes
(v_{0}\otimes w_{0})\omega (m_{1},v_{1},w_{1})
\end{eqnarray*}%
As in the previous proposition, $\mathcal{M}_{R}^{H}$ becomes a right $^{H}%
\mathcal{M}^{H}$-category with the above structures if and only if (\ref%
{weak action}), (\ref{unit}), (\ref{asociat}), (\ref{cocycle}), (\ref{vanish
cocycle}) hold. Again, this construction can be performed also for left $R$%
-modules left $H$-comodules, obtaining a right action of $^{H}\mathcal{M}%
^{H} $ on $_{R}^{H}\mathcal{M}$. We can summarize all these in the following
Theorem:

\begin{theorem}
Let $H$ be a coquasi-bialgebra and $(R,\cdot ,\sigma )$ an $H$-crossed
system. Then we have the following:

\begin{enumerate}
\item $\mathcal{M}_{R}$ is a right $^{H}\mathcal{M}$-category;

\item $_{R}\mathcal{M}$ is a right $\mathcal{M}^{H}$-category;

\item $\mathcal{M}_{R}^{H}$ is a right $^{H}\mathcal{M}^{H}$-category;

\item $_{R}^{H}\mathcal{M}$ is a right $^{H}\mathcal{M}^{H}$-category.
\end{enumerate}

In all four cases, the action is given by the usual tensor product.
\end{theorem}

As constructing a crossed system means giving an action of the monoidal
category, we want to see what is happening if we change the monoidal
category or the action. We shall treat only the case $\mathcal{M}_{R}$ is a
right $^{H}\mathcal{M}$-category, but obvious similar arguments work for the
other cases. First, consider an equivalence of monoidal categories. The
easiest way to do this is using a twist $\tau $ on $H$. Then from Remark (%
\ref{monoidal isomorfism gauge}) (left version) it follows that deforming
the multiplication of $H$ by conjugation gives us a new coquasi-bialgebra,
having category of comodules monoidal isomorphic with the starting category
of comodules $^{H}\mathcal{M}\simeq {}^{H_{\tau }}\mathcal{M}$. Via this
isomorphism, we obtain an action of the deformed category of comodules over
the category of $R$-modules. Moreover, the crossed system is changed and it
follows that the resulting crossed product comodule algebra is precisely the
image of the initial one via this monoidal category isomorphism (right
version, as we deal with right comodule algebras). Explicitly, we have:

\begin{proposition}
Let $H$ be a quasi-bialgebra, $\tau \in (H\otimes H)^{\ast }$ a twist on $H$
and $(R,\cdot ,\sigma )$ a crossed system. Then:

\begin{enumerate}
\item There is a right action of $^{H_{\tau }}\mathcal{M}$ on $\mathcal{M}%
_{R}$;

\item $(R,\cdot ,\sigma \tau ^{-1})$ is a crossed $H_{\tau }$-system;

\item $R\#_{\sigma \tau ^{-1}}H_{\tau }=(R\#_{\sigma }H)_{\tau ^{-1}}$.
\end{enumerate}
\end{proposition}

\begin{proof}
(1) Remember that the monoidal category isomorphism $^{H_{\tau }}\mathcal{M}%
\widetilde{\longrightarrow }\,^{H}\mathcal{M}$ is the identity on objects
and morphisms, but with monoidal structure $V\otimes W\longrightarrow
V\otimes W$, $v\otimes w\longrightarrow v_{0}\otimes w_{0}\tau
(v_{-1},w_{-1})$, $v\in V$, $w\in W$ and $V$, $W\in {}^{H_{\tau }}\mathcal{M}
$. Then simply a transport of structures gives us the action, whose changes
are reflected only in the reassociator's formula%
\begin{eqnarray*}
\Psi _{M,V,W} &:&(M\otimes V)\otimes W\longrightarrow M\otimes (V\otimes W)
\\
(m\otimes v)\otimes w &\longrightarrow &m\sigma (v_{1},w_{1})\tau
^{-1}(v_{2},w_{2})\otimes (v_{0}\otimes w_{0})
\end{eqnarray*}%
for $M\in \mathcal{M}_{R}$, $V$, $W\in {}^{H_{\tau }}\mathcal{M}$.

(2) It follows from the relations (\ref{multiplic Htau}), (\ref{asociator
Htau} ) defining the multiplication and the reassociator of the deformed
coquasi-bialgebra. Here we have denoted by $\sigma \tau ^{-1}$ the
convolution product.

(3) It is enough to write down explicitly the multiplication formulas for
both $(R\#_{\sigma }H)_{\tau ^{-1}}$ and $R\#_{\sigma \tau ^{-1}}H_{\tau }$
to conclude that they coincide.
\end{proof}

Now we want to change the action, by keeping both categories involved
unchanged. Let $H$ be a coquasi-bialgebra, $(R,\cdot ,\sigma )$ a crossed $H$%
-system and $\mathfrak{a}:H\longrightarrow R$ a convolution invertible map.
Define 
\begin{eqnarray}
h\cdot _{\mathfrak{a}}r &=&\mathfrak{a}^{-1}(h_{1})(h_{2}\cdot r)\mathfrak{a}%
(h_{3})  \label{deformed weak action} \\
\sigma _{\mathfrak{a}}(h,g) &=&\mathfrak{a}^{-1}(h_{1})[h_{2}\cdot \mathfrak{%
a}^{-1}(g_{1})]\sigma (h_{3},g_{2})\mathfrak{a}(h_{4}g_{3})
\label{deformed cocycle}
\end{eqnarray}%
for all $h,g\in H$, $r\in R$. Then we can easily check that (\ref{weak
action}), (\ref{unit}), (\ref{asociat}), (\ref{cocycle}), (\ref{vanish
cocycle}) hold for \textquotedblright $\cdot _{\mathfrak{a}}$%
\textquotedblright\ and $\sigma _{\mathfrak{a}}$, therefore we have:

\begin{proposition}
For any $\mathfrak{a}:H\longrightarrow R$ convolution invertible map, $%
(R,\cdot _{\mathfrak{a}},\sigma _{\mathfrak{a}})$ is again a crossed $H$%
-system.
\end{proposition}

The resulting crossed product $R\#_{\sigma _{\mathfrak{a}}}H$ will be called
a twist deformation of $R\#_{\sigma }H$. As before, we expect to obtain a
connection between these two structures. But we can say more in this case:
first, the twist transformation of the crossed product is isomorphic to the
initial one as left $R$-modules right $H$-comodules and second, any such
isomorphism is given by a twist transformation:

\begin{theorem}
\label{iso 2 crossed prod}Let $H$\ be a coquasi-bialgebra and $R$\ an
associative algebra. Consider on $R$ two crossed systems $(R,\cdot
_{1},\sigma _{1})$ and $(R,\cdot _{2},\sigma _{2})$. For any algebra
isomorphism $\theta :R\overline{\#}_{\sigma _{1}}H\longrightarrow R\overline{%
\#}_{\sigma _{2}}H$ left $R$-linear, right $H$-colinear, there is a
convolution invertible map $\mathfrak{a}:H\longrightarrow R$ such that:

\begin{enumerate}
\item $\theta (r\overline{\#}_{\sigma _{1}}h)=r\mathfrak{a}(h_{1})\overline{%
\#}_{\sigma _{2}}h_{2}$;

\item $h\cdot _{2}r=\mathfrak{a}^{-1}(h_{1})(h_{2}\cdot _{1}r)\mathfrak{a}%
(h_{3})$;

\item $\sigma _{2}(h,g)=\mathfrak{a}^{-1}(h_{1})[h_{2}\cdot _{1}\mathfrak{a}%
^{-1}(g_{1})]\sigma _{1}(h_{3},g_{2})\mathfrak{a}(h_{4}g_{3})$.
\end{enumerate}

Conversely, for each invertible map $\mathfrak{a}:H\longrightarrow R$ with
properties (2) and (3), the map $\theta $ given by (1) is an $H$-comodule
algebra isomorphism which is left $R$-linear, right $H$-colinear.
\end{theorem}

\begin{proof}
As in the Hopf case, see \cite{Montgomery93}, Theorem 7.3.4.
\end{proof}

\subsection{Coquasi-Hopf modules}

Let $H$ be a coquasi-bialgebra and $R$ an associative algebra, such that $%
(R,\cdot ,\sigma )$ is a crossed system. By the previous results, $\mathcal{M%
}_{R}^{H}$ is a right $^{H}\mathcal{M}^{H}$-category. But $H$ is an algebra
in $^{H}\mathcal{M}^{H}$, thus we may consider right modules over $H$ in the
category $\mathcal{M}_{R}^{H}$. This means

\begin{definition}
A right $\mathbf{(}R\mathbf{,}H\mathbf{)}$\textbf{-}coquasi-Hopf module $M$
is a right $R$-module, right $H$-comodule equipped with a map $\circ
:M\otimes H\longrightarrow M$, such that 
\begin{eqnarray}
\rho _{M}(m\circ h) &=&m_{0}\circ h_{1}\otimes m_{1}h_{2}\text{ (}R\text{%
-linearity)}  \label{r-linear} \\
(m\circ h)r &=&[m(h_{1}\cdot r)]\circ h_{2}\text{ (}H\text{-colinearity)}
\label{h colinear} \\
(m\circ h)\circ g &=&[m_{0}\sigma (h_{1},g_{1})]\circ (h_{2}g_{2})\omega
(m_{1},h_{3},g_{3})  \label{asociat h action} \\
m\circ 1_{H} &=&m
\end{eqnarray}%
for all $m\in M$, $h,g\in H$, $r\in R$. The category of right coquasi-Hopf
modules will be denoted\textbf{\ }$(\mathcal{M}_{R}^{H})_{H}$.
\end{definition}

Similarly, we may define the category of left coquasi-Hopf modules\ $%
(_{R}^{H}\mathcal{M)}_{H}$\ as the category of right $H$-modules within $%
_{R}^{H}\mathcal{M}$.

We shall see now the connection between the crossed product and the category
defined above. We start with a lemma.

\begin{lemma}
\begin{enumerate}
\item There is a functor 
\begin{equation*}
F:(\mathcal{M}_{R}^{H})_{H}\longrightarrow \mathcal{M}_{R\overline{\#}%
_{\sigma }H}^{H}
\end{equation*}%
defined as follows: it is identity on morphisms, and for each $(M,\rho
_{M})\in (\mathcal{M}_{R}^{H})_{H}$, we take $F(M)=M$, with structure maps 
\begin{equation*}
\left\{ 
\begin{array}{c}
\breve{\rho}_{M}(m)=\rho _{M}(m) \\ 
m\divideontimes (r\overline{\#}_{\sigma }h)=(mr)\circ h%
\end{array}%
\right.
\end{equation*}%
where $m\in M$, $r\in R$, $h\in H$.

\item We have a functor 
\begin{equation*}
G:\mathcal{M}_{R\overline{\#}_{\sigma }H}^{H}\longrightarrow (\mathcal{M}%
_{R}^{H})_{H}
\end{equation*}%
which acts as identity on morphisms, and for each $(M,\check{\rho}_{M})\in 
\mathcal{M}_{R\overline{\#}_{\sigma }H}^{H}$, we put $G(M)=M$, with
structure maps 
\begin{equation*}
\left\{ 
\begin{array}{c}
\rho _{M}(m)=\breve{\rho}_{M}(m) \\ 
mr=m\divideontimes (r\overline{\#}_{\sigma }1_{H}) \\ 
m\circ h=m\divideontimes (1_{R}\overline{\#}_{\sigma }h)%
\end{array}%
\right.
\end{equation*}%
where $m\in M$, $r\in R$, $h\in H$.
\end{enumerate}
\end{lemma}

\begin{proof}
It is easy to check that the above formulas define indeed two functors.
\end{proof}

\begin{theorem}
\label{structura coquasihopf modulelor}The category of right coquasi-Hopf
modules $(\mathcal{M}_{R}^{H})_{H}$ is isomorphic to the category of right $%
(H,R\overline{\#}_{\sigma }H)$-Hopf modules $\mathcal{M}_{R\overline{\#}%
_{\sigma }H}^{H}$.
\end{theorem}

\begin{proof}
It follows from the previous lemma. One has only to check that the above
correspondences are inverse to each other, which is almost immediate.
\end{proof}

\begin{remark}
\begin{enumerate}
\item In \cite{Schauenburg03}, Corollary 3.6 states the following:

Let $\mathcal{B}$ a strict monoidal category, $\mathcal{C}$ a monoidal
category, $R$ a coflat algebra in $\mathcal{B}$, $(\mathcal{B}_{R},\lozenge
) $ a right $\mathcal{C}$-category compatible with its natural $\mathcal{B}$%
-category structure, $A$ an algebra in $\mathcal{C}$ such that the functor $%
A\lozenge -:\mathcal{B}_{R}\longrightarrow \mathcal{B}_{R}$ preserves
coequalizers. Then $R\lozenge A$ is a $R$-ring in $\mathcal{B}$ and we have
a category equivalence $(\mathcal{B}_{R})_{A}\simeq \mathcal{B}_{R\lozenge
A} $.

Notice that if we relax the condition that $\mathcal{B}$ is strict (as any
monoidal category is equivalent to a strict one, by \cite{Maclane71}), and
take the particular case $\mathcal{B=M}^{H}$, $\mathcal{C=\,}^{H}\mathcal{M}%
^{H}$, $A=H$, then we get exactly our crossed product $R\lozenge A=R%
\overline{\#}_{\sigma }H$ as an algebra in $\mathcal{M}^{H}$, and recover
the category isomorphism $(\mathcal{M}_{R}^{H})_{H}\simeq \mathcal{M}_{R%
\overline{\#}_{\sigma }H}^{H}$\ of Theorem \ref{structura coquasihopf
modulelor}.

\item \label{eu si bulacu module in trei colturi}In \cite{Bulacu02}, a
similar notion of Hopf\ module was defined, but for $\mathfrak{H}$ a finite
dimensional quasi-Hopf algebra, and $R$ a right $\mathfrak{H}$-comodule
algebra, as in Section 3.2, example (4). Then the category of $(\mathfrak{H}%
,R)$-bimodules $_{\mathfrak{H}}\mathcal{M}_{R}$ is a right category over the
monoidal category (with tensor product over the base field) of $\mathfrak{H}$%
-bimodules $_{\mathfrak{H}}\mathcal{M}_{\mathfrak{H}}$. As $\mathfrak{H}$ is
a coalgebra in $_{\mathfrak{H}}\mathcal{M}_{\mathfrak{H}}$, the category of
two sided Hopf modules $_{\mathfrak{H}}\mathcal{M}_{R}^{\mathfrak{H}}$ was
defined as having the objects the right $\mathfrak{H}$-comodules in this
category and the morphisms the left $\mathfrak{H}$-linear, right $R$-linear,
right $\mathfrak{H}$-colinear maps. Moreover, it was proved that there is an
isomorphism of categories $_{\mathfrak{H}}\mathcal{M}_{R}^{\mathfrak{H}%
}\simeq {}_{\mathfrak{H}}\mathcal{M}_{R\overline{\#}\mathfrak{H}^{\ast }}$
(where $R\overline{\#}\mathfrak{H}^{\ast }$ is the quasi-smash product
mentioned also in Example (4) from Section 3.2). Now, take $H\mathcal{=}%
\mathfrak{H}^{\ast }$, as in the mentioned example. Then the acting monoidal
category is the same $_{\mathfrak{H}}\mathcal{M}_{\mathfrak{H}}={}^{H}%
\mathcal{M}^{H}$, the right $_{\mathfrak{H}}\mathcal{M}_{\mathfrak{H}}$%
-category is the same $_{\mathfrak{H}}\mathcal{M}_{R}=\mathcal{M}_{R}^{H}$
and the resulting categories are isomorphic, as they are both isomorphic to $%
\mathcal{M}_{R\overline{\#}_{\sigma }H}^{H}$. We do not write down
explicitly this isomorphism, as it implies some very complicated notations
not needed in this paper, but only remark that this isomorphism is
preserving the $R$-module structure, the $\mathfrak{H}$-module structure
being modified by $S^{2}$. Hence we cannot say that this category
isomorphism is induced by a duality between the coalgebra $\mathfrak{H}$ and
the algebra $H$ in the category $_{\mathfrak{H}}\mathcal{M}_{\mathfrak{H}}$
(notice that the right dual of $H$ -which is $\mathfrak{H}$ as vector space-
inherits a coalgebra structure within this monoidal category, with
comultiplication $\Delta _{r}(h)=(S^{-1}\otimes S^{-1})(f)\Delta
^{cop}(h)f^{(-1)21}$, where $f\in \mathfrak{H\otimes H}$ is the twist
introduced by Drinfeld in \cite{Drinfeld90}. The counit is $\varepsilon $,
while the bimodule structure of the dual is given by%
\begin{equation*}
h\rightharpoondown g=gS(h)\qquad g\leftharpoondown h=S^{-1}(h)g
\end{equation*}%
for $g,h\in \mathfrak{H}$, and%
\begin{eqnarray*}
db &:&\Bbbk \longrightarrow \mathfrak{H}^{\ast }\otimes \mathfrak{H}\qquad
db(1)=\sum_{i=1,\dim \mathfrak{H}}\mathfrak{e}_{i}^{\ast }\otimes
S^{-1}(\alpha )\mathfrak{e}_{i}\beta \\
ev &:&\mathfrak{H}\otimes \mathfrak{H}^{\ast }\longrightarrow \Bbbk \qquad
ev(h\otimes h^{\ast })=h^{\ast }(S^{-1}(\beta )h\alpha )
\end{eqnarray*}%
are the rigidity morphisms, with $(\mathfrak{e}_{i})_{i=1,\dim \mathfrak{H}}$
and $(\mathfrak{e}_{i}^{\ast })_{i=1,\dim \mathfrak{H}}$ dual bases for $%
\mathfrak{H}$, respectively $\mathfrak{H}^{\ast }$). It would be interesting
to find a categorical explanation of this isomorphism of categories.
Moreover, remark that an advantage of the formulas used in the present paper
is their naturality, compared with the ones in \cite{Bulacu02}, and they do
not involve any hard computations.
\end{enumerate}
\end{remark}

From now on we consider $H$ a coquasi-Hopf algebra and $(R,\cdot ,\sigma )$
a crossed system. Let $M$ a right coquasi-Hopf module and denote by $M^{coH}$
the subspace of coinvariants, namely $M^{coH}=\{m\in M\left\vert \rho
(m)=m\otimes 1_{H}\right\} $. Consider also the map $\Pi :M\longrightarrow M$%
, $\Pi (m)=m_{0}\circ S(m_{1}\leftharpoonup \beta )$. This map enjoys the
following properties:

\begin{proposition}
\label{proprietati proiectie}Under the above assumptions, we have the
following:

\begin{enumerate}
\item $\Pi ^{2}=\Pi $;

\item $\rho _{M}\Pi (m)=m\otimes 1_{H}$;

\item $\Pi (m)r=\Pi (m_{0}(S(m_{1})\cdot r))$, for any $m\in M$, $r\in R$.
\end{enumerate}
\end{proposition}

\begin{proof}
(1) Take $m\in M$. Then 
\begin{eqnarray*}
\Pi ^{2}(m) &=&\Pi (m_{0}\circ S(m_{1}\leftharpoonup \beta )) \\
&=&[m_{0}\circ \beta (m_{3})S(m_{6})]\circ \beta
(m_{1}S(m_{5}))S(m_{2}S(m_{4})) \\
&=&m_{0}\circ \beta (m_{2})S(m_{4})\beta (m_{1}S(m_{3})) \\
&=&m_{0}\circ \beta (m_{1})S(m_{2}) \\
&=&\Pi (m)
\end{eqnarray*}%
proving that $\Pi $ is a projection.

(2) Again, for $m\in M$ we have 
\begin{eqnarray*}
\rho \Pi (m) &=&\rho (m_{0}\circ \beta (m_{1})S(m_{2})) \\
&=&m_{0}\circ \beta (m_{2})S(m_{4})\otimes m_{1}S(m_{3}) \\
&=&m_{0}\circ \beta (m_{1})S(m_{2})\otimes 1_{H} \\
&=&\Pi (m)\otimes 1_{H}
\end{eqnarray*}

(3) We compute%
\begin{eqnarray*}
\Pi (m)r &=&(m_{0}\circ \beta (m_{1})S(m_{2}))r \\
&=&[m_{0}(S(m_{3})\cdot r)]\circ S(m_{2})\beta (m_{1}) \\
&=&\Pi (m_{0}(S(m_{1})\cdot r))
\end{eqnarray*}%
for any $m\in M$ and $r\in R$.
\end{proof}

\begin{corollary}
$M^{coH}=\Pi (M)$.
\end{corollary}

\begin{proof}
If $m\in M^{coH}$, then by the defining formula for $\Pi $ we have $\Pi
(m)=m $. The converse results from Proposition \ref{proprietati proiectie}%
.(2).
\end{proof}

Let $M$ a right coquasi-Hopf module. By Theorem 19, $M$ has a natural
structure of $(H,R\overline{\#}_{\sigma }H)$-Hopf module and from (2.20) it
follows that $M^{coH}$ inherits a structure of right module over $(R%
\overline{\#}_{\sigma }H)^{coH}=R\overline{\#}_{\sigma }\Bbbk 1_{H}\simeq R$.

\begin{corollary}
We have thus obtained a functor $(-)^{coH}:(\mathcal{M}_{R}^{H})_{H}%
\longrightarrow \mathcal{M}_{R}$.
\end{corollary}

For any $R$-module $N$, we may consider on $N\otimes H$ the structure of a
coquasi-Hopf module, by the following formulas:%
\begin{eqnarray*}
(n\otimes h)r &=&n(h_{1}\cdot r)\otimes h_{2} \\
\rho (n\otimes h) &=&n\otimes h_{1}\otimes h_{2} \\
(n\otimes h)\circ g &=&n\sigma (h_{1},g_{1})\otimes h_{2}g_{2}
\end{eqnarray*}

\begin{theorem}
\label{echivalenta cu r module}Let $H$ be a coquasi-Hopf algebra and $%
(R,\cdot ,\sigma )$ a crossed system with invertible cocycle. Then the
functors $-\otimes H$, $(-)^{coH}$ define a pair of inverse equivalences $%
\mathcal{M}_{R}\overset{-\otimes H}{\underset{(-)^{coH}}{\rightleftarrows }}(%
\mathcal{M}_{R}^{H})_{H}$.
\end{theorem}

\begin{proof}
For $M$ a coquasi-Hopf module, define the map $\varepsilon
_{M}:M^{coH}\otimes H\longrightarrow M$, $\varepsilon _{M}(m\otimes
h)=m\circ h$. Then $\varepsilon _{M}$ is a morphism in $(\mathcal{M}%
_{R}^{H})_{H}$, and it is natural in $M$.\ We need an inverse for $%
\varepsilon _{M}$. Define $\varkappa _{M}:M\longrightarrow M^{coH}\otimes H$%
, by $\varkappa _{M}(m)=\Pi (m_{0}\sigma ^{-1}(S(m_{1}),m_{2}\leftharpoonup
\alpha ))\otimes m_{3}$. Then for $m\in M^{coH}$, $h\in H$ we have:%
\begin{eqnarray*}
\varkappa _{M}\varepsilon _{M}(m\otimes h) &=&\varkappa _{M}(m\circ h) \\
&=&\Pi ((m_{0}\circ h_{1})\sigma
^{-1}(S(m_{1}h_{2}),(m_{2}h_{3})\leftharpoonup \alpha ))\otimes m_{3}h_{4} \\
&=&\Pi ((m\circ h_{1})\sigma ^{-1}(S(h_{2}),h_{3}\leftharpoonup \alpha
))\otimes h_{4} \\
&=&[(m\circ h_{1})_{0}\sigma ^{-1}(S(h_{2}),h_{3}\leftharpoonup \alpha
)]\circ \beta ((m\circ h_{1})_{1})S((m\circ h_{1})_{2})\otimes h_{4} \\
&=&[(m\circ h_{1})\sigma ^{-1}(S(h_{4}),h_{6})]\circ \beta
(h_{2})S(h_{3})\alpha (h_{5})\otimes h_{7} \\
&=&\{[m(h_{1}\cdot \sigma ^{-1}(S(h_{5}),h_{7}))]\circ h_{2}\}\circ \beta
(h_{3})S(h_{4})\alpha (h_{6})\otimes h_{8} \\
&=&[m(h_{1}\cdot \sigma ^{-1}(S(h_{9}),h_{11}))\sigma (h_{2},S(h_{8}))]\circ
(h_{3}S(h_{7}))\beta (h_{5})\alpha (h_{10})\omega (1,h_{4},S(h_{6}))\otimes
h_{12} \\
&=&m(h_{1}\cdot \sigma ^{-1}(S(h_{5}),h_{7}))\sigma (h_{2},S(h_{4}))\beta
(h_{3})\alpha (h_{6})\otimes h_{8} \\
(\ref{h ori sigma la -1}) &=&m\sigma (h_{1},S(h_{7})h_{9})\omega
(h_{2},S(h_{6}),h_{10})\sigma ^{-1}(h_{3}S(h_{5}),h_{11})\beta (h_{4})\alpha
(h_{8})\otimes h_{12} \\
&=&m\omega (h_{1},S(h_{3}),h_{5})\beta (h_{2})\alpha (h_{4})\otimes h_{6} \\
&=&m\otimes \varepsilon (h_{1})h_{2} \\
&=&m\otimes h
\end{eqnarray*}%
Conversely, for each $m\in M$, we compute 
\begin{eqnarray*}
\varepsilon _{M}\varkappa _{M}(m) &=&\Pi (m_{0}\sigma
^{-1}(S(m_{1}),m_{2}\leftharpoonup \alpha ))\circ m_{3} \\
&=&\{[m_{0}\sigma ^{-1}(S(m_{3}),\alpha (m_{4})m_{5})]\circ \beta
(m_{1})S(m_{2})\}\circ m_{6} \\
(\ref{asociat h action}) &=&[m_{0}\sigma ^{-1}(S(m_{6}),\alpha
(m_{7})m_{8})\sigma (S(m_{5}),m_{9})]\circ \beta
(m_{2})(S(m_{4})m_{10})\omega (m_{1},S(m_{3}),m_{11}) \\
&=&m_{0}\circ \alpha (m_{5})\beta (m_{2})(S(m_{4})m_{6})\omega
(m_{1},S(m_{3}),m_{7}) \\
&=&m_{0}\alpha (m_{4})\beta (m_{2})\omega (m_{1},S(m_{3}),m_{5}) \\
&=&m_{0}\varepsilon (m_{1}) \\
&=&m
\end{eqnarray*}

Next, for any right $R$-module $N$, define $u_{N}:N\longrightarrow (N\otimes
H)^{coH}$, by $u_{N}(n)=n\otimes 1_{H}$. It is easy to see that $u_{N}$ is
well defined, $R$-linear and natural in $N$. As in the Hopf case, we take $%
\upsilon _{N}:(N\otimes H)^{coH}\longrightarrow N$, $\upsilon
_{N}(\sum_{i}n_{i}\otimes h_{i})=\sum_{i}n_{i}\varepsilon (h_{i})$. Then $%
u_{N}$ and $\upsilon _{N}$ are inverses to each other.

We still need to check that $\varepsilon _{M}$ and $u_{N}$ make $F$ and $G$
a pair of adjoint factors.%
\begin{eqnarray*}
\varepsilon _{N\otimes H}(u_{N}\otimes I_{H})(\sum_{i}n_{i}\otimes h_{i})
&=&\sum_{i}(n_{i}\otimes 1_{H})\circ h_{i} \\
&=&\sum_{i}n_{i}\sigma (1_{H},h_{i1})\otimes h_{i2} \\
&=&\sum_{i}n_{i}\otimes h_{i}
\end{eqnarray*}%
and%
\begin{equation*}
\varepsilon _{M}u_{M^{coH}}(m)=m\circ 1_{H}=m
\end{equation*}%
Therefore, we have obtained the equivalence between the two categories.
\end{proof}

If we compose the previous equivalence with the isomorphism from Theorem \ref%
{structura coquasihopf modulelor}, we obtain exactly the adjunction between
the induced and the coinvariant functor from \cite{Balan07}. Therefore, if
the coquasi-bialgebra admits an antipode and the comodule algebra is a
crossed productwith invertible cocycle, the two above categories are
equivalent. We shall see in Section \ref{crossed=cleft} why is this
happening.

\section{Cleft extensions for coquasi-Hopf algebras}

\subsection{Cleft extensions from the Morita theory point of view}

One of the main results in the theory of Hopf algebras is the equivalent
characterization of cleft extensions as crossed product algebras with
invertible cocycle (cf. \cite{Blattner89}, \cite{Doi86}). In order to derive
such a characterization for coquasi-Hopf algebras, we need an appropriate
definition of a cleft extension.

Let $H$ be a coquasi-Hopf algebra and $A$ a right $H$-comodule algebra.
Denote $B=A^{coH}$. We recall from \cite{Balan07} the following definition:

\begin{definition}
\label{cleft definition}Let $A$ a right $H$-comodule algebra and $\gamma
:H\longrightarrow A$ a colinear map. The extension $B\subseteq A$ is cleft
with respect to the cleaving map $\gamma $ if there is a linear map $\delta
:H\longrightarrow A$ such that 
\begin{eqnarray}
\rho \delta (h) &=&\delta (h_{2})\otimes S(h_{1})  \label{inversecleaving} \\
\delta (h_{1})\gamma (h_{2}) &=&\alpha (h)1_{A}  \label{convolutiedeltagama}
\\
\gamma (h_{1})\beta (h_{2})\delta (h_{3}) &=&\varepsilon (h)1_{A}
\label{convolutiegamabetadelta}
\end{eqnarray}%
In this case, we call the pair $(\gamma ,\delta )$ a cleaving system for the
extension $B\subseteq A$.
\end{definition}

\begin{remark}
\begin{enumerate}
\item \label{observatia cleft}This definition of cleftness is slightly
different from the classical one, where it is only required that $\gamma $
is convolution invertible (denote by $\delta $ the convolution inverse of $%
\gamma $) and $H$-colinear. The property (\ref{inversecleaving}) appears
naturally by passing from a bialgebra to a Hopf algebra. Unfortunately, in
our case the convolution product on $Hom(H,A)$ is no longer associative,
therefore a left inverse for $\gamma $ is not necessarily a right inverse
and the property (\ref{inversecleaving}) does not seem to result from the
other properties of $\gamma $. So we had to state it separately.

\item For a cleft comodule algebra $A$, the application $\delta $ and
relations (\ref{convolutiedeltagama}), (\ref{convolutiegamabetadelta})
depend on the antipode, again unlike the classical case. But if we change
the antipode and the linear maps $\alpha $, $\beta $ to $(S^{\prime },\alpha
^{\prime },\beta ^{\prime })$ as in (\ref{change antipode coquasi}) and
define $\delta ^{\prime }(h)=U(h_{1})\delta (h_{2})$, then it follows
immediately that $A$ is also $H$-cleft, but with respect to the new
antipode. In the sequel, we shall suppose the antipode and the elements $%
\alpha $, $\beta $ fixed once for all.
\end{enumerate}
\end{remark}

In \cite{Balan07}, the above conditions imposed on the cleaving map $\gamma $
were stated without further explanations, the only motivation being the
equivalence with Galois extensions with normal basis property. We shall see
now that the relations (\ref{inversecleaving})-(\ref{convolutiegamabetadelta}%
) come from a Morita context and that this is the reason for their
non-symmetry. The following construction was inspired from \cite{Bohm07},
where the coring case was treated.

First of all, notice that if $\mathcal{C}$ is a $\Bbbk $-linear monoidal
category, $(A,\mu _{A},u_{A})$ an algebra and $(C,\Delta _{C},\varepsilon
_{C})$ a coalgebra in this monoidal category, then $Hom_{\mathcal{C}}(C,A)$
becomes an associative $\Bbbk $-algebra, with multiplication $\varphi \ast
\psi =\mu _{A}(\varphi \otimes \psi )\Delta _{C}$ and unit $u_{A}\varepsilon
_{C}$. Take now $\mathcal{C}=\mathcal{M}^{H}$ the category of comodules over
a coquasi-Hopf algebra and $A$ a right $H$-comodule algebra. We also need a
coalgebra in this monoidal category. In \cite{Bulacu02co}, the authors
deformed the comultiplication on $H$ in order to obtain a left $H$-comodule
coalgebra (actually, they showed that for any coalgebra $C$ with coalgebra
map $C\longrightarrow H$, there is a structure of left $H$-comodule
coalgebra on $C$). Repeating their argument, but for $H^{op,cop}$ this time,
we obtain that $H$ is a coalgebra in $\mathcal{M}^{H}$, with the following
structures: right adjoint coaction $\overline{\rho }(h)=h_{2}\otimes
S(h_{1})h_{3}$, comultiplication 
\begin{equation*}
\overline{\Delta }(h)=h_{3}\otimes h_{9}\omega
(S(h_{2})h_{4},S(h_{8}),h_{10})\beta (h_{6})\omega
^{-1}(S(h_{1}),h_{5},S(h_{7}))
\end{equation*}%
and counit $\overline{\varepsilon }(h)=\alpha (h)$, for any $h\in H$. We
shall denote by $\overline{H}$ this new structure. Hence $Hom^{H}(\overline{H%
},A)$ becomes an associative algebra, with multiplication and unit as
described above. As $B=A^{coH}$ is an associative algebra, $Hom(H,B)$ will
also be with the usual convolution product. We have the two rings for the
Morita context, we need the connecting bimodules. One of them will be $%
Hom^{H}(H,A)$, where $H$ is a comodule via $\Delta $, and the other $%
Hom^{H}(H^{S},A)$. Here we have denoted by $H^{S}$ the comodule structure of 
$H$ twisted by the antipode, namely $h\longrightarrow h_{2}\otimes S(h_{1})$.

\begin{lemma}
$Hom^{H}(H,A)$ becomes a $(Hom(H,B),Hom^{H}(\overline{H},A))$-bimodule with
the following structures:%
\begin{eqnarray}
(\mathfrak{rp})(h) &=&\mathfrak{r}(h_{1})\mathfrak{p}(h_{2})
\label{Hom(H,B)*HomH(H,A)} \\
(\mathfrak{ps})(h) &=&\mathfrak{p}(h_{1})\mathfrak{s}(h_{5})\beta
(h_{3})\omega (h_{2},S(h_{4}),h_{6})  \label{HomH(H,A)*HomH(Hbar,A)}
\end{eqnarray}%
for $h\in H$, $\mathfrak{p}\in Hom^{H}(H,A)$, $\mathfrak{s}\in Hom^{H}(%
\overline{H},A)$ and $\mathfrak{r}\in Hom(H,B)$.
\end{lemma}

\begin{proof}
It is easy to see that the formula (\ref{Hom(H,B)*HomH(H,A)}) defines a
structure of $Hom(H,B)$-module on $Hom^{H}(H,A)$. We need now to verify that 
$\mathfrak{ps}$ is $H$-colinear, for any $\mathfrak{s}\in Hom^{H}(\overline{H%
},A)$ and $\mathfrak{p}\in Hom^{H}(H,A)$:%
\begin{eqnarray*}
(\mathfrak{ps})(h)_{0}\otimes (\mathfrak{ps})(h)_{1} &=&\mathfrak{p}%
(h_{1})_{0}\mathfrak{s}(h_{5})_{0}\otimes \mathfrak{p}(h_{1})_{1}\mathfrak{s}%
(h_{5})_{1}\beta (h_{3})\omega (h_{2},S(h_{4}),h_{6}) \\
&=&\mathfrak{p}(h_{1})\mathfrak{s}(h_{7})\otimes h_{2}(S(h_{6})h_{8})\beta
(h_{4})\omega (h_{3},S(h_{5}),h_{9}) \\
(\ref{asociat multipl}),(\ref{IdbetaS}) &=&\mathfrak{p}(h_{1})\mathfrak{s}%
(h_{5})\otimes h_{7}\beta (h_{3})\omega (h_{2},S(h_{4}),h_{6}) \\
&=&(\mathfrak{ps})(h_{1})\otimes h_{2}
\end{eqnarray*}%
For any $\mathfrak{s},\overline{\mathfrak{s}}\in Hom^{H}(\overline{H},A)$, $%
\mathfrak{p}\in Hom^{H}(H,A)$ and $h\in H$, we compute that%
\begin{eqnarray*}
(\mathfrak{p}(\mathfrak{s}\ast \overline{\mathfrak{s}}))(h) &=&\mathfrak{p}%
(h_{1})(\mathfrak{s}\ast \overline{\mathfrak{s}})(h_{5})\beta (h_{3})\omega
(h_{2},S(h_{4}),h_{6}) \\
&=&\mathfrak{p}(h_{1})(\mathfrak{s}(h_{7})\overline{\mathfrak{s}}%
(h_{13}))\omega (S(h_{6})h_{8},S(h_{12}),h_{14})\beta (h_{10})\omega
^{-1}(S(h_{5}),h_{9},S(h_{11})) \\
&&\beta (h_{3})\omega (h_{2},S(h_{4}),h_{15}) \\
(\ref{asociat multipl}) &=&(\mathfrak{p}(h_{1})\mathfrak{s}(h_{9}))\overline{%
\mathfrak{s}}(h_{17})\omega
^{-1}(h_{2},S(h_{8})h_{10},S(h_{16})h_{18})\omega
(S(h_{7})h_{11},S(h_{15}),h_{19}) \\
&&\beta (h_{13})\omega ^{-1}(S(h_{6}),h_{12},S(h_{14}))\beta (h_{4})\omega
(h_{3},S(h_{5}),h_{20}) \\
(\ref{cocycle omega}),(\ref{IdbetaS}) &=&(\mathfrak{p}(h_{1})\mathfrak{s}%
(h_{8}))\overline{\mathfrak{s}}(h_{16})\omega
^{-1}(h_{2},S(h_{7})h_{9},S(h_{15})h_{17})\omega
^{-1}(S(h_{6}),h_{10},S(h_{14})h_{18}) \\
&&\beta (h_{12})\omega (h_{11},S(h_{13}),h_{19})\beta (h_{4})\omega
(h_{3},S(h_{5}),h_{20}) \\
(\ref{cocycle omega}),(\ref{asociat multipl}),(\ref{IdbetaS}) &=&(\mathfrak{p%
}(h_{1})\mathfrak{s}(h_{11}))\overline{\mathfrak{s}}(h_{18})\omega
(h_{2},S(h_{10}),h_{12})\omega ^{-1}(h_{3}S(h_{9}),h_{13},S(h_{17})h_{19}) \\
&&\omega ^{-1}(h_{4},S(h_{8}),h_{21})\beta (h_{15})\omega
(h_{14},S(h_{16}),h_{20})\beta (h_{6})\omega (h_{5},S(h_{7}),h_{22}) \\
(\ref{IdbetaS}) &=&(\mathfrak{p}(h_{1})\mathfrak{s}(h_{5}))\overline{%
\mathfrak{s}}(h_{10})\omega (h_{2},S(h_{4}),h_{6})\omega
(h_{7},S(h_{9}),h_{11})\beta (h_{3})\beta (h_{8}) \\
&=&((\mathfrak{ps})\overline{\mathfrak{s}})(h)
\end{eqnarray*}%
and $(\mathfrak{p}\alpha )(h)=\mathfrak{p}(h_{1})\alpha (h_{5})\beta
(h_{3})\omega (h_{2},S(h_{4}),h_{6})=\mathfrak{p}(h)$. Hence $Hom^{H}(H,A)$
is a right $Hom^{H}(\overline{H},A)$-module, and it is easy to check now the
compatibility between the two module structures.
\end{proof}

\begin{lemma}
$Hom^{H}(H^{S},A)$ becomes a $(Hom^{H}(\overline{H},A),Hom(H,B))$-bimodule
with the following structures:%
\begin{eqnarray}
(\mathfrak{sq})(h) &=&\mathfrak{s}(h_{2})\mathfrak{q}(h_{6})\beta
(h_{4})\omega ^{-1}(S(h_{1}),h_{3},S(h_{5}))  \label{HomH(Hbar,A)*HomH(HS,A)}
\\
(\mathfrak{qr})(h) &=&\mathfrak{q}(h_{1})\mathfrak{r}(h_{2})
\label{HomH(HS,A)*Hom(H,B)}
\end{eqnarray}%
for $h\in H$, $\mathfrak{q}\in Hom^{H}(H^{S},A)$, $\mathfrak{s}\in Hom^{H}(%
\overline{H},A)$ and $\mathfrak{r}\in Hom(H,B)$.
\end{lemma}

\begin{proof}
As in the previous Lemma, the only difficult part\ to check is the left $%
Hom^{H}(\overline{H},A)$-module structure. For this, let $h\in H$, $%
\mathfrak{q}\in Hom^{H}(H^{S},A)$, $\mathfrak{s}\in Hom^{H}(\overline{H},A)$
and compute%
\begin{eqnarray*}
(\mathfrak{sq})_{0}(h)\otimes (\mathfrak{sq})_{0}(h) &=&\mathfrak{s}%
(h_{2})_{0}\mathfrak{q}(h_{6})_{0}\otimes \mathfrak{s}(h_{2})_{0}\mathfrak{q}%
(h_{6})_{0}\beta (h_{4})\omega ^{-1}(S(h_{1}),h_{3},S(h_{5})) \\
&=&\mathfrak{s}(h_{3})\mathfrak{q}(h_{9})\otimes
(S(h_{2})h_{4})S(h_{8})\beta (h_{6})\omega ^{-1}(S(h_{1}),h_{5},S(h_{7})) \\
(\ref{asociat multipl}),(\ref{IdbetaS}) &=&\mathfrak{s}(h_{3})\mathfrak{q}%
(h_{7})\otimes S(h_{1})\beta (h_{5})\omega ^{-1}(S(h_{2}),h_{4},S(h_{6})) \\
&=&(\mathfrak{sq})(h_{2})\otimes S(h_{1})
\end{eqnarray*}%
Therefore the action of $Hom^{H}(\overline{H},A)$ on $\mathfrak{q}\in
Hom^{H}(H^{S},A)$ is well defined. Take now $\mathfrak{s},\overline{%
\mathfrak{s}}\in Hom^{H}(\overline{H},A)$, $\mathfrak{q}\in Hom^{H}(H^{S},A)$
and $h\in H$. Then%
\begin{allowdisplaybreaks}
\begin{eqnarray*}
((\mathfrak{s}\ast \overline{\mathfrak{s}})\mathfrak{q})(h) &=&(\mathfrak{s}%
\ast \overline{\mathfrak{s}})(h_{2})\mathfrak{q}(h_{6})\beta (h_{4})\omega
^{-1}(S(h_{1}),h_{3},S(h_{5})) \\
&=&(\mathfrak{s}(h_{4})\overline{\mathfrak{s}}(h_{10}))\mathfrak{q}%
(h_{15})\omega (S(h_{3})h_{5},S(h_{9}),h_{11})\beta (h_{7})\omega
^{-1}(S(h_{2}),h_{6},S(h_{8}))\beta (h_{13}) \\
&&\omega ^{-1}(S(h_{1}),h_{12},S(h_{14})) \\
(\ref{asociat multipl}) &=&\mathfrak{s}(h_{5})(\overline{\mathfrak{s}}%
(h_{13})\mathfrak{q}(h_{20}))\omega
(S(h_{4})h_{6},S(h_{12})h_{14},S(h_{19}))\omega
(S(h_{3})h_{7},S(h_{11}),h_{15}) \\
&&\beta (h_{9})\omega ^{-1}(S(h_{2}),h_{8},S(h_{10}))\beta (h_{17})\omega
^{-1}(S(h_{1}),h_{16},S(h_{18})) \\
(\ref{cocycle omega}),(\ref{asociat multipl}),(\ref{IdbetaS}) &=&\mathfrak{s}%
(h_{5})(\overline{\mathfrak{s}}(h_{12})\mathfrak{q}(h_{22}))\omega
^{-1}(S(h_{11}),h_{13},S(h_{21}))\omega
(S(h_{4})h_{6},S(h_{10}),h_{14}S(h_{20})) \\
&&\beta (h_{8})\omega (S(h_{2}),h_{15},S(h_{19}))\beta (h_{17})\omega
^{-1}(S(h_{3}),h_{7},S(h_{9}))\omega ^{-1}(S(h_{1}),h_{16},S(h_{18})) \\
(\ref{IdbetaS}) &=&\mathfrak{s}(h_{2})(\overline{\mathfrak{s}}(h_{7})%
\mathfrak{q}(h_{11}))\omega ^{-1}(S(h_{6}),h_{8},S(h_{10}))\beta
(h_{4})\omega ^{-1}(S(h_{1}),h_{3},S(h_{5}))\beta (h_{9}) \\
&=&\mathfrak{s}(h_{2})(\overline{\mathfrak{s}}\mathfrak{q})(h_{6}))\beta
(h_{4})\omega ^{-1}(S(h_{1}),h_{3},S(h_{5})) \\
&=&(\mathfrak{s}(\overline{\mathfrak{s}}\mathfrak{q}))(h)
\end{eqnarray*}%
\end{allowdisplaybreaks}%
From (\ref{omega anihileaza S}) it follows that $(\mathfrak{s}\alpha )(h)=%
\mathfrak{s}(h)$, hence $Hom^{H}(H^{S},A)$ is a left $Hom^{H}(\overline{H}%
,A) $-module.
\end{proof}

\begin{proposition}
We have a Morita context 
\begin{equation*}
\mathbb{M}(A)=(Hom^{H}(\overline{H}%
,A),Hom(H,B),Hom^{H}(H,A),Hom^{H}(H^{S},A),(-,-),[-,-])
\end{equation*}%
with connecting morphisms%
\begin{eqnarray}
(-,-) &:&Hom^{H}(H,A)\otimes _{Hom^{H}(\overline{H},A)}Hom^{H}(H^{S},A)%
\longrightarrow Hom(H,B)  \notag \\
(\mathfrak{p},\mathfrak{q})(h) &=&\mathfrak{p}(h_{1})\beta (h_{2})\mathfrak{q%
}(h_{3})  \label{(-,-)} \\
\lbrack -,-] &:&Hom^{H}(H^{S},A)\otimes
_{Hom(H,B)}Hom^{H}(H,A)\longrightarrow Hom^{H}(\overline{H},A)  \notag \\
\lbrack \mathfrak{q},\mathfrak{p}](h) &=&\mathfrak{q}(h_{1})\mathfrak{p}%
(h_{2})  \label{[-,-]}
\end{eqnarray}
\end{proposition}

\begin{proof}
It is not difficult to see that $(\mathfrak{p},\mathfrak{q})\in Hom(H,B)$, $[%
\mathfrak{q},\mathfrak{p}]\in Hom^{H}(\overline{H},A)$, $[-,-]$ is $Hom(H,B)$%
-balanced and $(-,-)$ is $Hom(H,B)$ bilinear. All the remaining
verifications involve the algebra $Hom^{H}(\overline{H},A)$, and we shall do
them in detail, for the convenience of the reader.

We show first that $(-,-)$ is $Hom^{H}(\overline{H},A)$-balanced: 
\begin{eqnarray*}
(\mathfrak{ps},\mathfrak{q})(h) &=&(\mathfrak{ps})(h_{1})\beta (h_{2})%
\mathfrak{q}(h_{3}) \\
&=&(\mathfrak{p}(h_{1})\mathfrak{s}(h_{5}))\mathfrak{q}(h_{8})\beta
(h_{3})\omega (h_{2},S(h_{4}),h_{6})\beta (h_{7}) \\
(\ref{asociat multipl}) &=&\mathfrak{p}(h_{1})(\mathfrak{s}(h_{7})\mathfrak{q%
}(h_{12}))\omega (h_{2},S(h_{6})h_{8},S(h_{11}))\beta (h_{4})\omega
(h_{3},S(h_{5}),h_{9})\beta (h_{10}) \\
(\ref{cocycle omega}),(\ref{IdbetaS}),(\ref{IdbetaS}) &=&\mathfrak{p}(h_{1})(%
\mathfrak{s}(h_{4})\mathfrak{q}(h_{8}))\beta (h_{2})\omega
^{-1}(S(h_{3}),h_{5},S(h_{7}))\beta (h_{6}) \\
&=&\mathfrak{p}(h_{1})\beta (h_{2})(\mathfrak{sq})(h_{3}) \\
&=&(\mathfrak{p},\mathfrak{sq})(h)
\end{eqnarray*}%
Next, we check the $Hom^{H}(\overline{H},A)$-bilinearity of $[-,-]$:%
\begin{eqnarray*}
\lbrack \mathfrak{sq},\mathfrak{p}](h) &=&(\mathfrak{sq})(h_{1})\mathfrak{p}%
(h_{2}) \\
&=&(\mathfrak{s}(h_{2})\mathfrak{q}(h_{6}))\mathfrak{p}(h_{7})\beta
(h_{4})\omega ^{-1}(S(h_{1}),h_{3},S(h_{5})) \\
(\ref{asociat multipl}) &=&\mathfrak{s}(h_{3})(\mathfrak{q}(h_{9})\mathfrak{p%
}(h_{10}))\omega (S(h_{2})h_{4},S(h_{8}),h_{11})\beta (h_{6})\omega
^{-1}(S(h_{1}),h_{5},S(h_{7})) \\
&=&\mathfrak{s}(h_{3})[\mathfrak{q},\mathfrak{p}](h_{9})\omega
(S(h_{2})h_{4},S(h_{8}),h_{10})\beta (h_{6})\omega
^{-1}(S(h_{1}),h_{5},S(h_{7})) \\
&=&(\mathfrak{s}[\mathfrak{q},\mathfrak{p}])(h)
\end{eqnarray*}%
and%
\begin{eqnarray*}
\lbrack \mathfrak{q},\mathfrak{ps}](h) &=&\mathfrak{q}(h_{1})(\mathfrak{ps}%
)(h_{2}) \\
&=&\mathfrak{q}(h_{1})(\mathfrak{p}(h_{2})\mathfrak{s}(h_{6}))\beta
(h_{4})\omega (h_{3},S(h_{5}),h_{7}) \\
(\ref{asociat multipl}) &=&(\mathfrak{q}(h_{2})\mathfrak{p}(h_{3}))\mathfrak{%
s}(h_{9})\omega ^{-1}(S(h_{1}),h_{4},S(h_{8})h_{10})\beta (h_{6})\omega
(h_{5},S(h_{7}),h_{11}) \\
(\ref{cocycle omega}),(\ref{IdbetaS}) &=&(\mathfrak{q}(h_{3})\mathfrak{p}%
(h_{4}))\mathfrak{s}(h_{10})\omega (S(h_{2})h_{5},S(h_{9}),h_{11})\beta
(h_{7})\omega ^{-1}(S(h_{1}),h_{6},S(h_{8})) \\
&=&[\mathfrak{q},\mathfrak{p}](h_{3})\mathfrak{s}(h_{9})\omega
(S(h_{2})h_{4},S(h_{8}),h_{10})\beta (h_{6})\omega
^{-1}(S(h_{1}),h_{5},S(h_{7})) \\
&=&([\mathfrak{q},\mathfrak{p}]\mathfrak{s})(h)
\end{eqnarray*}%
for any $\mathfrak{s}\in Hom^{H}(\overline{H},A)$, $\mathfrak{p}\in
Hom^{H}(H,A)$, $\mathfrak{q}\in Hom^{H}(H^{S},A)$ and $h\in H$. Finally, we
compute 
\begin{eqnarray*}
((\mathfrak{p},\mathfrak{q})\overline{\mathfrak{p}})(h) &=&(\mathfrak{p},%
\mathfrak{q})(h_{1})\overline{\mathfrak{p}}(h_{2}) \\
&=&(\mathfrak{p}(h_{1})\mathfrak{q}(h_{3}))\overline{\mathfrak{p}}%
(h_{4})\beta (h_{2}) \\
(\ref{asociat multipl}) &=&\mathfrak{p}(h_{1})(\mathfrak{q}(h_{5})\overline{%
\mathfrak{p}}(h_{6}))\omega (h_{2},S(h_{4}),h_{7})\beta (h_{3}) \\
&=&\mathfrak{p}(h_{1})[\mathfrak{q},\overline{\mathfrak{p}}](h_{5})\omega
(h_{2},S(h_{4}),h_{6})\beta (h_{3}) \\
&=&(\mathfrak{p}[\mathfrak{q},\overline{\mathfrak{p}}])(h)
\end{eqnarray*}%
and%
\begin{eqnarray*}
([\mathfrak{q},\mathfrak{p}]\overline{\mathfrak{q}})(h) &=&[\mathfrak{q},%
\mathfrak{p}](h_{2})\overline{\mathfrak{q}}(h_{6})\beta (h_{4})\omega
^{-1}(S(h_{1}),h_{3},S(h_{5})) \\
&=&(\mathfrak{q}(h_{2})\mathfrak{p}(h_{3}))\overline{\mathfrak{q}}%
(h_{7})\beta (h_{5})\omega ^{-1}(S(h_{1}),h_{4},S(h_{6})) \\
(\ref{asociat multipl}) &=&\mathfrak{q}(h_{3})(\mathfrak{p}(h_{4})\overline{%
\mathfrak{q}}(h_{10}))\omega (S(h_{2}),h_{5},S(h_{9}))\beta (h_{7})\omega
^{-1}(S(h_{1}),h_{6},S(h_{8})) \\
&=&\mathfrak{q}(h_{1})(\mathfrak{p}(h_{2})\overline{\mathfrak{q}}%
(h_{4}))\beta (h_{3}) \\
&=&\mathfrak{q}(h_{1})(\mathfrak{p},\overline{\mathfrak{q}})(h_{2}) \\
&=&(\mathfrak{q}(\mathfrak{p},\overline{\mathfrak{q}}))(h)
\end{eqnarray*}%
where $\mathfrak{p},\overline{\mathfrak{p}}\in Hom^{H}(H,A)$, $\mathfrak{q},%
\overline{\mathfrak{q}}\in Hom^{H}(H^{S},A)$ and $h\in H$.
\end{proof}

\begin{remark}
\label{lema bohm}We shall denote by $_{B}\mathcal{M}^{H}$ the category of
left $B$-modules, right $H$-comodules $(M,\rho _{M})$ such that $\rho
_{M}(bm)=bm_{0}\otimes m_{1}$ for all $b\in B$, $m\in M$. The morphisms are
the left $B$-linear, right $H$-colinear maps. Two objects in this category
are $A$ and $B\otimes H$ with obvious structures. Then $_{B}\mathcal{M}^{H}$%
\ can be seen as a category of entwined modules, with trivial left-right
entwining structure. Therefore $B^{op}\otimes H$ is a $B^{op}$-coring, and $%
_{B}\mathcal{M}^{H}$ is precisely the category of right comodules over this
coring. We shall use in the proof of the next theorem the Lemma 3.5 from 
\cite{Bohm07}, applied to our situation.
\end{remark}

We are now able to see the relationship between the above Morita context and
cleft extensions. The following theorem is the coquasi-Hopf version of
Theorem 3.6 from \cite{Bohm07}:

\begin{theorem}
Let $H$ a coquasi-Hopf algebra and $A$ a right $H$-comodule algebra. Then:

\begin{enumerate}
\item The map $[,]$ is surjective if and only if $B\subseteq A$ is Galois
and there is an nonnegative integer $n$ such that $A$ is a direct summand in 
$(B\otimes H)^{n}$ as left $B$-module, right $H$-comodule.

\item The Morita context is strict if and only if $B\subseteq A$ is Galois
and there are nonnegative integers $n$, $\overline{n}$ such that $A$ is a
direct summand in $(B\otimes H)^{n}$ and $B\otimes H$ is direct summand in $%
A^{\overline{n}}$ as left $B$-modules, right $H$-comodules.

\item If $B\subseteq A$ is cleft, then the above Morita context is strict.
\end{enumerate}
\end{theorem}

\begin{proof}
(1) If $[,]$ is surjective, choose $\mathfrak{p}_{i}\in Hom^{H}(H,A)$, $%
\mathfrak{q}_{i}\in Hom^{H}(H^{S},A)$, for $i\in I$ a finite index set such
that $\sum_{i\in I}[\mathfrak{q}_{i},\mathfrak{p}_{i}]=\alpha 1_{A}$. Take
the map $\Upsilon :A\otimes H\longrightarrow A\otimes _{B}A$, $\Upsilon
(a\otimes h)=\sum_{i\in I}a\mathfrak{q}_{i}(h_{1})\otimes _{B}\mathfrak{p}%
_{i}(h_{2})$. We claim that $\Upsilon $ is the inverse of $can$. Indeed, for
any $a\in A$ and $h\in H$, we have%
\begin{allowdisplaybreaks}%
\begin{eqnarray*}
can(\sum_{i\in I}a\mathfrak{q}_{i}(h_{1})\otimes _{B}\mathfrak{p}%
_{i}(h_{2})) &=&\sum_{i\in I}[a_{0}\mathfrak{q}_{i}(h_{1})_{0}]\mathfrak{p}%
_{i}(h_{2})_{0}\otimes \mathfrak{p}_{i}(h_{2})_{4} \\
&&\omega ^{-1}(a_{1}\mathfrak{q}_{i}(h_{1})_{1},\mathfrak{p}%
_{i}(h_{2})_{1},\beta (\mathfrak{p}_{i}(h_{2})_{2})S(\mathfrak{p}%
_{i}(h_{2})_{3})) \\
(\mathfrak{p}_{i}\in Hom^{H}(H,A)\text{, }\mathfrak{q}_{i}\in
Hom^{H}(H^{S},A)) &=&\sum_{i\in I}[a_{0}\mathfrak{q}_{i}(h_{2})]\mathfrak{p}%
_{i}(h_{3})\otimes h_{7}\omega ^{-1}(a_{1}S(h_{1}),h_{4},\beta
(h_{5})S(h_{6})) \\
(\ref{asoc comod alg}),(\mathfrak{p}_{i}\in Hom^{H}(H,A)\text{, }\mathfrak{q}%
_{i}\in Hom^{H}(H^{S},A)) &=&\sum_{i\in I}a_{0}[\mathfrak{q}_{i}(h_{3})%
\mathfrak{p}_{i}(h_{4})]\otimes h_{9}\omega (a_{1},S(h_{2}),h_{5}) \\
&&\omega ^{-1}(a_{2}S(h_{1}),h_{6},\beta (h_{7})S(h_{8})) \\
(\sum_{i\in I}[\mathfrak{q}_{i},\mathfrak{p}_{i}]=\alpha 1_{A})
&=&\sum_{i\in I}a_{0}\alpha (h_{3})\otimes h_{8}\omega (a_{1},S(h_{2}),h_{4})
\\
&&\omega ^{-1}(a_{2}S(h_{1}),h_{5},\beta (h_{6})S(h_{7})) \\
(\ref{cocycle omega}) &=&\sum_{i\in I}a_{0}\otimes h_{12}\omega
^{-1}(a_{1},S(h_{3})\alpha (h_{4})h_{5},S(h_{11})) \\
&&\omega ^{-1}(S(h_{2}),h_{6},S(h_{10}))\omega (a_{2},S(h_{1}),h_{7}\beta
(h_{8})S(h_{9})) \\
(\ref{IdbetaS}),(\ref{SalfaId}),(\ref{omega anihileaza S}) &=&a\otimes h
\end{eqnarray*}

\end{allowdisplaybreaks}%
On the other hand, for any $a,b\in A$,%
\begin{align*}
\Upsilon can(a\otimes _{B}b)& =\Upsilon (a_{0}b_{0}\otimes b_{4}\omega
^{-1}(a_{1},b_{1}\beta (b_{2}),S(b_{3}))) \\
& =\sum_{i\in I}(a_{0}b_{0})\mathfrak{q}_{i}(b_{4})\otimes _{B}\mathfrak{p}%
_{i}(b_{5})\omega ^{-1}(a_{1},b_{1}\beta (b_{2}),S(b_{3})) \\
(\ref{asoc comod alg}),(\mathfrak{q}_{i}\in Hom^{H}(H^{S},A))& =\sum_{i\in
I}a_{0}[b_{0}\mathfrak{q}_{i}(b_{6})]\otimes _{B}\mathfrak{p}%
_{i}(b_{7})\omega (a_{1},b_{1},S(b_{5}))\omega ^{-1}(a_{2},b_{2}\beta
(b_{3}),S(b_{4})) \\
& =\sum_{i\in I}a[b_{0}\mathfrak{q}_{i}(b_{2})]\otimes _{B}\mathfrak{p}%
_{i}(b_{3})\beta (b_{1})
\end{align*}%
But for any $b\in A$,%
\begin{eqnarray}
\rho _{A}(\sum_{i\in I}b_{0}\beta (b_{1})\mathfrak{q}_{i}(b_{2}))
&=&\sum_{i\in I}b_{0}\mathfrak{q}_{i}(b_{3})_{0}\otimes b_{1}\beta (b_{2})%
\mathfrak{q}_{i}(b_{3})_{1}  \notag \\
(\mathfrak{q}_{i}\in Hom^{H}(H^{S},A)) &=&\sum_{i\in I}b_{0}\mathfrak{q}%
_{i}(b_{4})\otimes b_{1}\beta (b_{2})S(b_{3})  \notag \\
(\ref{IdbetaS}) &=&\sum_{i\in I}b_{0}\beta (b_{1})\mathfrak{q}%
_{i}(b_{2})\otimes 1_{H}  \label{trace map}
\end{eqnarray}%
from where it follows that%
\begin{allowdisplaybreaks}
\begin{eqnarray}
\Upsilon can(a\otimes _{B}b) &=&\sum_{i\in I}a[b_{0}\mathfrak{q}%
_{i}(b_{2})]\otimes _{B}\mathfrak{p}_{i}(b_{3})\beta (b_{1})
\label{can-1 compus cu can} \\
&=&\sum_{i\in I}a\otimes _{B}[b_{0}\mathfrak{q}_{i}(b_{2})]\mathfrak{p}%
_{i}(b_{3})\beta (b_{1})  \notag \\
(\ref{asoc comod alg}),(\mathfrak{p}_{i}\in Hom^{H}(H,A)\text{, }\mathfrak{q}%
_{i}\in Hom^{H}(H^{S},A)) &=&\sum_{i\in I}a\otimes _{B}b_{0}[\mathfrak{q}%
_{i}(b_{3})\mathfrak{p}_{i}(b_{4})]\beta (b_{1})\omega (b_{1},S(b_{2}),b_{5})
\notag \\
(\sum_{i\in I}[\mathfrak{q}_{i},\mathfrak{p}_{i}]=\alpha 1_{A}) &=&a\otimes
_{B}b_{0}\alpha (b_{3})\beta (b_{1})\omega (b_{1},S(b_{2}),b_{4})  \notag \\
(\ref{omega anihileaza S}) &=&a\otimes _{B}b  \notag
\end{eqnarray}

\end{allowdisplaybreaks}%
Hence the extension is Galois.

We want now to prove the second statement. Consider the following maps, for
every $i\in I$:%
\begin{eqnarray}
\zeta _{i} &:&A\longrightarrow B\otimes H,\quad \zeta _{i}(a)=\sum_{i\in
I}a_{0}\beta (a_{1})\mathfrak{q}_{i}(a_{2})\otimes a_{3}  \label{zeta} \\
\xi _{i} &:&B\otimes H\longrightarrow A,\quad \xi _{i}(b\otimes h)=b%
\mathfrak{p}_{i}(h)  \label{xi}
\end{eqnarray}%
The maps $\zeta _{i}$ are well defined from (\ref{trace map}), left $B$%
-linear and right $H$-colinear, as it can be easily checked. Also $\xi _{i}$
are left $B$-linear and right $H$-colinear and $\sum_{i\in I}\xi _{i}(\zeta
_{i}(a))=\sum_{i\in I}[a_{0}\beta (a_{1})\mathfrak{q}_{i}(a_{2})]\mathfrak{p}%
_{i}(a_{3})=a$, after a similar computation as in (\ref{can-1 compus cu can}%
). It follows from Remark \ref{lema bohm} that $A$ is a direct summand in $%
(B\otimes H)^{n}$ as left $B$-module, right $H$-comodule, where $n$ is the
cardinal of the index set $I$.

Conversely, if $A$ is a direct summand in $(B\otimes H)^{n}$ then again by
Remark \ref{lema bohm} there exist some morphisms $\zeta _{i}\in
Hom_{B}^{H}(A,B\otimes H)$, $\xi _{i}\in Hom_{B}^{H}(B\otimes H,A)$, $i\in I$
with $I$ a finite index set, $\left\vert I\right\vert =n$, such that $%
\sum_{i\in I}\xi _{i}\zeta _{i}=I_{A}$. Define then 
\begin{eqnarray}
\mathfrak{p}_{i} &:&H\longrightarrow A,\quad \mathfrak{p}_{i}(h)=\xi
_{i}(1_{A}\otimes h)  \label{p} \\
\mathfrak{q}_{i} &:&H\longrightarrow A,\quad \mathfrak{q}_{i}(h)=(I_{A}%
\otimes _{B}(I_{B}\otimes \varepsilon ))(I_{A}\otimes _{B}\zeta
_{i})can^{-1}(1_{A}\otimes h)  \label{q}
\end{eqnarray}

It is immediate that $\mathfrak{p}_{i}\in Hom^{H}(H,A)$, $\forall i\in I$.
In order to show that $\mathfrak{q}_{i}\in Hom^{H}(H^{S},A)$, we start by
recalling from \cite{Balan07} the notation $can^{-1}(1_{A}\otimes
h)=\sum_{j}l_{j}(h)\otimes _{B}r_{j}(h)$ and the following properties:%
\begin{allowdisplaybreaks}%
\begin{eqnarray}
\sum\limits_{j}l_{j}(h)_{0}\otimes _{B}r_{j}(h)\otimes l_{j}(h)_{1}
&=&\sum\limits_{j}l_{j}(h_{2})\otimes _{B}r_{j}(h_{2})\otimes S(h_{1})
\label{l(h)} \\
\sum\limits_{j}l_{j}(h_{1})\otimes _{B}r_{j}(h_{1})\otimes h_{2}
&=&\sum\limits_{j}l_{j}(h)\otimes _{B}r_{j}(h)_{0}\otimes r_{j}(h)_{1}
\label{r(h)} \\
\sum\limits_{j}l_{j}(h)r_{j}(h) &=&\alpha (h)1_{A}  \label{lr} \\
\sum_{j}al_{j}(h)\otimes _{B}r_{j}(h) &=&can^{-1}(a\otimes h)
\label{altensorr=can-1} \\
a_{0}\otimes \beta (a_{1})a_{2} &=&can(1_{A}\otimes _{B}a)
\label{can(1tensora)}
\end{eqnarray}%
\end{allowdisplaybreaks}%
for any $h\in H$, $a\in A$. Using the identification $A\otimes _{B}B\simeq A$%
, we may write $\mathfrak{q}_{i}(h)=\sum_{j}l_{j}(h)(I_{B}\otimes
\varepsilon )\zeta _{i}(r_{j}(h))$ and the $H$-colinearity of $\mathfrak{q}%
_{i}$ follows from (\ref{l(h)}).

We compute now, for $h\in H$%
\begin{allowdisplaybreaks}%
\begin{eqnarray*}
\sum_{i}[\mathfrak{q}_{i},\mathfrak{p}_{i}](h) &=&\sum_{i}\mathfrak{q}%
_{i}(h_{1})\mathfrak{p}_{i}(h_{2}) \\
&=&\sum_{i,j}l_{j}(h_{1})\underset{\in B}{\underbrace{(I_{B}\otimes
\varepsilon )\zeta _{i}(r_{j}(h_{1}))}}\xi _{i}(1_{A}\otimes h_{2}) \\
&=&\sum_{i,j}l_{j}(h_{1})\xi _{i}((I_{B}\otimes \varepsilon )\zeta
_{i}(r_{j}(h_{1}))\otimes h_{2}) \\
&=&\sum_{i,j}l_{j}(h_{1})[\xi _{i}(I_{B}\otimes \varepsilon \otimes
I_{H})(\zeta _{i}(r_{j}(h_{1}))\otimes h_{2})] \\
(\ref{r(h)}) &=&\sum_{i,j}l_{j}(h)[\xi _{i}(I_{B}\otimes \varepsilon \otimes
I_{H})(\zeta _{i}(r_{j}(h)_{0})\otimes r_{j}(h)_{1})] \\
&=&\sum_{i,j}l_{j}(h)[\xi _{i}(I_{B}\otimes \varepsilon \otimes
I_{H})(I_{B}\otimes \Delta )\zeta _{i}(r_{j}(h))] \\
&=&\sum_{i,j}l_{j}(h)[\xi _{i}\zeta _{i}(r_{j}(h))] \\
&=&\sum_{j}l_{j}(h)r_{j}(h)\overset{(\ref{lr})}{=}\alpha (h)1_{A}
\end{eqnarray*}

\end{allowdisplaybreaks}%
This proves the surjectivity of $[,]$.

(2) Suppose that the Morita context is strict. By (1), we have only to show
that $B\otimes H$ is direct summand in $A^{\overline{n}}$ for some integer $%
\overline{n}$. In order to do this, we repeat the arguments from (1), but
now for the Morita map $(,)$. Therefore we may find $\overline{\mathfrak{p}}%
_{i}\in Hom^{H}(H,A)$, $\overline{\mathfrak{q}}_{i}\in Hom^{H}(H^{S},A)$,
for $i\in \overline{I}$ a finite index set such that $\sum_{i}(\overline{%
\mathfrak{p}}_{i},\overline{\mathfrak{q}}_{i})=\varepsilon 1_{A}$. We define 
$\overline{\zeta }_{i}\in Hom_{B}^{H}(A,B\otimes H)$, $\overline{\xi }%
_{i}\in Hom_{B}^{H}(B\otimes H,A)$ by similar formulas to (\ref{zeta})-(\ref%
{xi}), but with $\overline{\mathfrak{p}}_{i}$, $\overline{\mathfrak{q}}_{i}$
instead of $\mathfrak{p}_{i}$, $\mathfrak{q}_{i}$ . The linearity and
colinearity of them follow easily. Then we find that, for any $b\in B$ and $%
h\in H$, we have%
\begin{allowdisplaybreaks}%
\begin{eqnarray*}
\sum_{i\in I}\zeta _{i}(\xi _{i}(b\otimes h)) &=&\sum_{i\in I}\zeta _{i}(b%
\overline{\mathfrak{p}}_{i}(h)) \\
&=&\sum_{i\in I}b\zeta _{i}(\overline{\mathfrak{p}}_{i}(h))=b\sum_{i}%
\overline{\mathfrak{p}}_{i}(h)_{0}\beta (\overline{\mathfrak{p}}_{i}(h)_{1})%
\overline{\mathfrak{q}}_{i}(\overline{\mathfrak{p}}_{i}(h)_{2})\otimes 
\overline{\mathfrak{p}}_{i}(h)_{3} \\
&=&b\sum_{i}\overline{\mathfrak{p}}_{i}(h_{1})\beta (h_{2})\overline{%
\mathfrak{q}}_{i}(h_{3})\otimes h_{4} \\
&=&b\otimes h
\end{eqnarray*}%
\end{allowdisplaybreaks}%
As in (1), it follows that $B\otimes H$ is direct summand in $A^{\overline{n}%
}$ for $\overline{n}=\left\vert \overline{I}\right\vert $.

For the converse statement, again by (1) we need only to check the
surjectivity of the Morita map $(,)$. Similar to (1), from the fact that $%
B\otimes H$ is direct summand in $A^{\overline{n}}$, it follows the
existence of a finite index set $\overline{I}$ with $\left\vert \overline{I}%
\right\vert =\overline{n}$ and of two families of morphisms $\overline{\zeta 
}_{i}\in Hom_{B}^{H}(A,B\otimes H)$, $\overline{\xi }_{i}\in
Hom_{B}^{H}(B\otimes H,A)$, $i\in \overline{I}$ such that $\sum_{i}\zeta
_{i}\xi _{i}=I_{B}\otimes I_{H}$. Again use formulas (\ref{p})-(\ref{q}) to
define $\overline{\mathfrak{p}}_{i}\in Hom^{H}(H,A)$, $\overline{\mathfrak{q}%
}_{i}\in Hom^{H}(H^{S},A)$ by means of $\overline{\zeta }_{i}$ and $%
\overline{\xi }_{i}$. Then 
\begin{eqnarray*}
\sum_{i}(\overline{\mathfrak{p}}_{i},\overline{\mathfrak{q}}_{i})(h)
&=&\sum_{i}\overline{\mathfrak{p}}_{i}(h_{1})\beta (h_{2})\overline{%
\mathfrak{q}}_{i}(h_{3}) \\
&=&\sum_{i,j}\overline{\mathfrak{p}}_{i}(h_{1})\beta (h_{2})l_{j}(h_{3})%
\underset{\in B}{\underbrace{(I_{B}\otimes \varepsilon )\overline{\zeta }%
_{i}(r_{j}(h_{3}))}}
\end{eqnarray*}%
But relation (\ref{altensorr=can-1}) implies%
\begin{eqnarray*}
\sum_{i,j}\overline{\mathfrak{p}}_{i}(h_{1})\beta (h_{2})l_{j}(h_{3})\otimes
_{B}r_{j}(h_{3}) &=&\sum_{i}can^{-1}(\overline{\mathfrak{p}}%
_{i}(h_{1})\otimes \beta (h_{2})h_{3}) \\
(\ref{can(1tensora)}) &=&1_{A}\otimes _{B}\overline{\mathfrak{p}}_{i}(h)
\end{eqnarray*}%
where the last equality uses the fact that $\overline{\mathfrak{p}}_{i}$ are 
$H$-colinear, $\forall i\in \overline{I}$. Therefore 
\begin{eqnarray*}
\sum_{i}(\overline{\mathfrak{p}}_{i},\overline{\mathfrak{q}}_{i})(h)
&=&\sum_{i}(I_{B}\otimes \varepsilon )\overline{\zeta }_{i}\overline{%
\mathfrak{p}}_{i}(h) \\
&=&\sum_{i}(I_{B}\otimes \varepsilon )\overline{\zeta }_{i}\overline{\xi }%
_{i}(1_{A}\otimes h) \\
&=&\varepsilon (h)1_{A}
\end{eqnarray*}%
and the surjectivity of $(,)$ follows from its $Hom(H,B)$-bilinearity.

(3) Suppose $B\subseteq A$ is cleft with maps $\gamma $, $\delta
:H\longrightarrow A$, where $\gamma \in Hom^{H}(H,A)$. But (\ref%
{inversecleaving}) implies that $\delta \in Hom^{H}(H^{S},A)$, while
properties (\ref{convolutiedeltagama})-(\ref{convolutiegamabetadelta}) are
equivalent to $[\delta ,\gamma ]=\alpha 1_{A}$, respectively to $(\gamma
,\delta )=\varepsilon 1_{A}$. Therefore the cleaving maps are sent by the
Morita morphisms (\ref{(-,-)}), (\ref{[-,-]})\ to the unit elements of the
corresponding algebras in the Morita context. It follows that the Morita
context is strict.
\end{proof}

\begin{remark}
\begin{enumerate}
\item Notice that in the Hopf algebra case, all algebras and bimodules
involved in the above Morita context are included in $Hom(H,A)$, and all
structure maps and connecting homomorphisms are precisely the convolution
product. This was observed in \cite{Bohm07a} for the coring case, but it
remains true for coquasi-Hopf algebras. We can say even more in this case.
There is a second Morita context that we may build generalizing Doi's
construction (\cite{Doi94}). We plan to investigate the relationship between
these two Morita contexts in a forthcoming paper.

\item The Morita context which inspired us (\cite{Bohm07a}) has a very
natural conceptual meaning, it is simply given by the natural
transformations of two functors between comodule categories over some
corings. It is unclear for the moment how this should be applied to the
present situation, mainly because $A\otimes H$ is no longer an $A$-coring in
the usual sense, as $A$ is not an associative algebra.
\end{enumerate}
\end{remark}

\subsection{\label{crossed=cleft}Crossed products are the same as cleft
extensions by a coquasi-Hopf algebra}

Recall that in \cite{Balan07}, the notion of a Galois extension for a
coquasi-Hopf algebra was introduced, and it was also proven the equivalence
between cleft extensions and Galois extensions with the normal basis
pro\-perty, under the additional hypothesis of the bijectivity of the
antipode. This is a generalization of a well-known result for Hopf algebras.
Also, in the Hopf algebra theory, cleft extensions of Hopf algebras can be
\linebreak characterized as crossed products with invertible cocycle by Hopf
algebras. We shall see that this identification remains true for the
coquasi-Hopf algebras:

\begin{theorem}
\label{doiTakBlattnerMontg crossed}Let $H$ a coquasi-Hopf algebra, $A$ a
right $H$-comodule algebra and $B=A^{coH}$ the subalgebra of coinvariants.
The following statements are equivalent:

(1) The extension $B\subseteq A$ is cleft;

(2) There exist an invertible cocycle $\sigma $ and a weak action of $H$ on $%
B$ such that $A$ is isomorphic as left $B$-module, right $H$-comodule
algebra with the crossed product $B\overline{\#}_{\sigma }H$.
\end{theorem}

\begin{proof}
(1) $\Longrightarrow $ (2) In \cite{Balan07}, an isomorphism $\nu
:{}B\otimes H\longrightarrow A$, left $B$-linear and right $H$-colinear was
constructed by the formulas 
\begin{eqnarray}
\nu (b\otimes h) &=&b\gamma (h)  \label{normal basis izo} \\
\nu ^{-1}(a) &=&a_{0}\delta (a_{1}\leftharpoonup \beta )\otimes a_{2}
\label{normal basis inverse}
\end{eqnarray}%
Via this isomorphism, $B\otimes H$ becomes an algebra in $\mathcal{M}^{H}$,
with multiplication%
\begin{equation*}
(b\otimes h)\circledcirc (c\otimes g)=\nu ^{-1}(\nu (b\otimes h)\nu
(c\otimes g))
\end{equation*}%
and unit $\nu ^{-1}(1_{A})$. But relations (\ref{convolutiedeltagama}), (\ref%
{convolutiegamabetadelta}) imply $\gamma (1_{H})\in B$ and invertible in $B$%
, hence we may assume $\gamma (1_{H})=1_{A}$ (if not, replace $\gamma $ by $%
\overline{\gamma }(h)=\gamma (h)\gamma (1_{H})^{-1}$ and $\delta $ by $%
\overline{\delta }(h)=\gamma (1_{H})\delta (h)$). It implies $\nu
^{-1}(1_{A})=1_{A}\otimes 1_{H}$ and $\nu (b\otimes 1_{H})=b$, $\forall b\in
B$. Now the rest of the proof follows as in \cite{Schauenburg04}. Define%
\begin{eqnarray*}
h\cdot b &=&(I_{B}\otimes \varepsilon )\nu ^{-1}(\nu (1_{A}\otimes h)\nu
(b\otimes 1_{H})) \\
\sigma (h,g) &=&(I_{B}\otimes \varepsilon )\nu ^{-1}(\nu (1_{A}\otimes h)\nu
(1_{R}\otimes g))
\end{eqnarray*}%
for any $h,g\in H$, $b\in B$. Then%
\begin{eqnarray*}
h_{1}\cdot b\otimes h_{2} &=&(I_{B}\otimes \varepsilon )\nu ^{-1}(\nu
(1_{A}\otimes h_{1})\nu (b\otimes 1_{H}))\otimes h_{2} \\
&=&(I_{B}\otimes \varepsilon \otimes I_{H})(\nu ^{-1}\otimes I_{H})(\nu
(1_{A}\otimes h_{1})\nu (b\otimes 1_{H})\otimes h_{2}) \\
(\nu \text{ is colinear}) &=&(I_{B}\otimes \varepsilon \otimes I_{H})(\nu
^{-1}\otimes I_{H})(\nu (1_{A}\otimes h)_{0}\nu (b\otimes 1_{H})\otimes \nu
(1_{A}\otimes h)_{1}) \\
(\nu ^{-1}\text{ is colinear}) &=&(I_{B}\otimes \varepsilon \otimes
I_{H})(I_{B}\otimes \Delta )\nu ^{-1}(\nu (1_{A}\otimes h)\nu (b\otimes
1_{H})) \\
&=&\nu ^{-1}(\nu (1_{A}\otimes h)\nu (b\otimes 1_{H})) \\
&=&(1_{A}\otimes h)\circledcirc (b\otimes 1_{H})
\end{eqnarray*}%
and%
\begin{eqnarray*}
\sigma (h_{1},g_{1})\otimes h_{2}g_{2} &=&(I_{B}\otimes \varepsilon )\nu
^{-1}(\nu (1_{A}\otimes h_{1})\nu (1_{R}\otimes g_{1}))\otimes h_{2}g_{2} \\
&=&(I_{B}\otimes \varepsilon \otimes I_{H})(\nu ^{-1}\otimes I_{H})(\nu
(1_{A}\otimes h_{1})\nu (1_{R}\otimes g_{1})\otimes h_{2}g_{2}) \\
(\nu \text{ is colinear}) &=&(I_{B}\otimes \varepsilon \otimes I_{H})(\nu
^{-1}\otimes I_{H})(\nu (1_{A}\otimes h)_{0}\nu (1_{R}\otimes g)_{0}\otimes
\nu (1_{A}\otimes h)_{1} \\
&&\nu (1_{R}\otimes g)_{1}) \\
(\nu ^{-1}\text{ is colinear}) &=&(I_{B}\otimes \varepsilon \otimes
I_{H})(I_{B}\otimes \Delta )\nu ^{-1}(\nu (1_{A}\otimes h)\nu (1_{A}\otimes
g)) \\
&=&\nu ^{-1}(\nu (1_{A}\otimes h)\nu (1_{A}\otimes g)) \\
&=&(1_{A}\otimes h)\circledcirc (1_{A}\otimes g)
\end{eqnarray*}%
Then we may compute that%
\begin{eqnarray*}
b(h_{1}\cdot c)\sigma (h_{2},g_{1})\otimes h_{3}g_{2} &=&b(h_{1}\cdot
c)[\sigma (h_{2},g_{1})\otimes h_{3}g_{2}] \\
&=&b(h_{1}\cdot c)[\nu ^{-1}(\nu (1_{A}\otimes h)\nu (1_{A}\otimes g))] \\
(\nu \text{ is }B\text{-linear}) &=&\nu ^{-1}([b(h_{1}\cdot c)]\nu
(1_{A}\otimes h)\nu (1_{A}\otimes g)) \\
&=&\nu ^{-1}(b\nu (h_{1}\cdot c\otimes h)\nu (1_{A}\otimes g)) \\
&=&\nu ^{-1}(b[\nu (1_{A}\otimes h)\nu (c\otimes 1_{H})]\nu (1_{A}\otimes g))
\\
&=&\nu ^{-1}(b\nu (1_{A}\otimes h)c\nu (1_{A}\otimes g)) \\
&=&\nu ^{-1}(\nu (b\otimes h)\nu (c\otimes g)) \\
&=&(b\otimes h)\circledcirc (c\otimes g)
\end{eqnarray*}%
It results that the multiplication formula for $B\otimes H$ is precisely the
one for the crossed product. As we have defined the multiplication on $%
B\otimes H$ such that it is comodule algebra, Theorem \ref{conditii produs
incrucisat}\ implies that "$\cdot $" and $\sigma $ verify the requested
relations (\ref{weak action}), (\ref{unit}), (\ref{asociat}), (\ref{vanish
cocycle}), (\ref{cocycle}).

If we write down explicitly the morphisms $\nu $ and $\nu ^{-1}$ from
relations (\ref{normal basis izo}) and (\ref{normal basis inverse}), we
obtain that the weak $H$-action on $B$ and the cocycle $\sigma $ are given
by the formulas%
\begin{eqnarray*}
h\cdot b &=&\gamma (h_{1})b\delta (h_{2}\leftharpoonup \beta ) \\
\sigma (h,g) &=&[\gamma (h_{1})\gamma (g_{1})]\delta
((h_{2}g_{2})\leftharpoonup \beta )
\end{eqnarray*}%
which are, up to the element $\beta \in H^{\ast }$, precisely as in the Hopf
algebra case. For the proof to be complete, we have to show that $\sigma $
is convolution invertible. Define for any $h,g\in H$, 
\begin{equation*}
\sigma ^{-1}(h,g)=\gamma (\beta \rightharpoonup
(h_{1}g_{1}))f^{-1}(h_{2},g_{2})[\delta (g_{3})\delta (h_{3})]
\end{equation*}%
Then this is coinvariant with respect to the $H$-coaction:%
\begin{eqnarray*}
\rho _{A}(\sigma ^{-1}(h,g)) &=&\gamma (h_{1}g_{1})_{0}\beta
(h_{2}g_{2})f^{-1}(h_{3},g_{3})[\delta (g_{4})_{0}\delta (h_{4})_{0}]\otimes
\gamma (h_{1}g_{1})_{1}[\delta (g_{4})_{1}\delta (h_{4})_{1}] \\
(\ref{inversecleaving}) &=&\gamma (h_{1}g_{1})\beta
(h_{3}g_{3})f^{-1}(h_{4},g_{4})[\delta (g_{6})\delta (h_{6})]\otimes
(h_{2}g_{2})[S(g_{5})S(h_{5})] \\
(\ref{twist f}),(\ref{IdbetaS}) &=&\gamma (h_{1}g_{1})\beta
(h_{2}g_{2})f^{-1}(h_{3},g_{3})[\delta (g_{4})\delta (h_{4})]\otimes 1_{H} \\
&=&\sigma ^{-1}(h,g)\otimes 1_{H}
\end{eqnarray*}%
We shall now compute the convolution product%
\begin{allowdisplaybreaks}%
\begin{align*}
\sigma (h_{1},g_{1})\sigma ^{-1}(h_{2},g_{2})& =\{[\gamma (h_{1})\gamma
(g_{1})]\delta (h_{3}g_{3})\}\{\gamma (h_{4}g_{4})[\delta (g_{7})\delta
(h_{7})]\} \\
& \beta (h_{2}g_{2})\beta (h_{5}g_{5})f^{-1}(h_{6},g_{6}) \\
(\ref{asoc comod alg}),(\ref{inversecleaving})& =[\gamma (h_{1})\gamma
(g_{1})]\{\delta (h_{5}g_{5})[\gamma (h_{6}g_{6})(\delta (g_{11})\delta
(h_{11}))]\} \\
& \omega (h_{2}g_{2},S(h_{4}g_{4}),(h_{7}g_{7})[S(g_{10})S(h_{10})]) \\
& \beta (h_{3}g_{3})\beta (h_{8}g_{8})f^{-1}(h_{9},g_{9}) \\
(\ref{twist f}),(\ref{IdbetaS})& =[\gamma (h_{1})\gamma (g_{1})]\{\delta
(h_{3}g_{3})[\gamma (h_{4}g_{4})(\delta (g_{7})\delta (h_{7}))]\} \\
& \beta (h_{2}g_{2})\beta (h_{5}g_{5})f^{-1}(h_{6},g_{6}) \\
(\ref{asoc comod alg}),(\ref{inversecleaving}),(\ref{convolutiedeltagama}),(%
\ref{twist f})& =[\gamma (h_{1})\gamma (g_{1})][\delta (g_{9})\delta
(h_{9})]\omega ^{-1}(S(h_{3}g_{3},h_{5}g_{5},S(h_{7}g_{7})) \\
& \beta (h_{2}g_{2})\alpha (h_{4}g_{4})\beta (h_{6}g_{6})f^{-1}(h_{8},g_{8})
\\
(\ref{omega anihileaza S})& =[\gamma (h_{1})\gamma (g_{1})][\delta
(g_{4})\delta (h_{4})]\beta (h_{2}g_{2})f^{-1}(h_{3},g_{3}) \\
(\ref{asoc comod alg}),(\ref{asoc comod alg})& =\gamma (h_{1})\{[\gamma
(g_{1})\delta (g_{8})]\delta (h_{7})\}\beta (h_{3}g_{4})f^{-1}(h_{4},g_{5})
\\
& \omega (h_{2},g_{3},S(g_{6})S(h_{5}))\omega ^{-1}(g_{2},S(g_{7}),S(h_{6}))
\\
(\ref{1beta*f-1=delta})& =\gamma (h_{1})\{[\gamma (g_{1})\delta
(g_{8})]\delta (h_{7})\}\beta (h_{5})\beta (g_{6})\omega
(h_{3}g_{4},S(g_{8}),S(h_{6})) \\
& \omega ^{-1}(h_{4},g_{5},S(g_{7}))\omega
(h_{2},g_{3},S(g_{9})S(h_{7}))\omega ^{-1}(g_{2},S(g_{10}),S(h_{8})) \\
(\ref{cocycle omega})& =\gamma (h_{1})\{[\gamma (g_{1})\delta (g_{5})]\delta
(h_{5})\}\beta (h_{3})\beta (g_{3})\omega (h_{2},g_{2}S(g_{4}),S(h_{4})) \\
(\ref{IdbetaS}),(\ref{convolutiegamabetadelta}),(\ref%
{convolutiegamabetadelta})& =\varepsilon (g)\varepsilon (h)1_{A}
\end{align*}%
\end{allowdisplaybreaks}%
The similar computations for showing that $\sigma ^{-1}(h_{1},g_{1})\sigma
(h_{2},g_{2})=\varepsilon (g)\varepsilon (h)1_{A}$, for $h,g\in H$ are left
to the reader. Therefore $\sigma $ is invertible with respect to the
convolution product.

(2) $\Longrightarrow $ (1) Define $\gamma :H\longrightarrow B\overline{\#}%
_{\sigma }H$ by $\gamma (h)=1_{A}\overline{\#}_{\sigma }h$. It is $H$%
-colinear, because the comodule structure on the crossed product is $%
I_{A}\otimes \Delta $. We need an inverse\ for $\gamma $ (in the sense of
Definition 25). If we denote $\delta (h)=\sigma
^{-1}(S(h_{2}),h_{3}\leftharpoonup \alpha )\overline{\#}_{\sigma }S(h_{1})$
for $h\in H$, then%
\begin{eqnarray*}
\rho _{A}(\delta (h)) &=&\sigma ^{-1}(S(h_{2}),h_{3}\leftharpoonup \alpha )%
\overline{\#}_{\sigma }S(h_{1})_{1}\otimes S(h_{1})_{2} \\
&=&\sigma ^{-1}(S(h_{3}),h_{4}\leftharpoonup \alpha )\overline{\#}_{\sigma
}S(h_{2})\otimes S(h_{1}) \\
&=&\delta (h_{2})\otimes S(h_{1})
\end{eqnarray*}%
We can now compute that:%
\begin{eqnarray*}
\delta (h_{1})\gamma (h_{2}) &=&[\sigma ^{-1}(S(h_{2}),h_{3}\leftharpoonup
\alpha )\overline{\#}_{\sigma }S(h_{1})](1_{A}\overline{\#}_{\sigma }h_{4})
\\
&=&\sigma ^{-1}(S(h_{2}),h_{3}\leftharpoonup \alpha )\sigma
(S(h_{1})_{1},h_{4_{1}})\overline{\#}_{\sigma }S(h_{1})_{2}h_{4_{2}} \\
&=&1_{A}\overline{\#}_{\sigma }S(h_{1})\alpha (h_{2})h_{3} \\
&=&1_{A}\overline{\#}_{\sigma }1_{H}\alpha (h)
\end{eqnarray*}%
and%
\begin{eqnarray*}
\gamma (h_{1})\beta (h_{2})\delta (h_{3}) &=&(1\overline{\#}_{\sigma
}h_{1}\beta (h_{2}))(\sigma ^{-1}(S(h_{4}),h_{5}\leftharpoonup \alpha )%
\overline{\#}_{\sigma }S(h_{3})) \\
&=&[h_{1_{1}}\beta (h_{2})\cdot \sigma ^{-1}(S(h_{4}),h_{5}\leftharpoonup
\alpha )]\sigma (h_{1_{2}},S(h_{3})_{1})\overline{\#}_{\sigma
}h_{1_{3}}S(h_{3})_{2} \\
&=&[h_{1}\beta (h_{4})\cdot \sigma ^{-1}(S(h_{7}),h_{8}\leftharpoonup \alpha
)]\sigma (h_{2},S(h_{6}))\overline{\#}_{\sigma }h_{3}S(h_{5}) \\
&=&[h_{1}\beta (h_{3})\cdot \sigma ^{-1}(S(h_{5}),h_{6}\leftharpoonup \alpha
)]\sigma (h_{2},S(h_{4}))\overline{\#}_{\sigma }1_{H} \\
(\ref{h ori sigma la -1}) &=&\sigma (h_{1_{1}},S(h_{5})_{1}h_{7_{1}})\omega
(h_{1_{2}},S(h_{5})_{2},h_{7_{2}})\sigma
^{-1}(h_{1_{3}}S(h_{5})_{3},h_{7_{3}})\sigma ^{-1}(h_{1_{4}},S(h_{5})_{4}) \\
&&\sigma (h_{2},S(h_{4}))\overline{\#}_{\sigma }1_{H}\alpha (h_{6})\beta
(h_{3}) \\
&=&\omega (h_{1},S(h_{5}),h_{7})\sigma ^{-1}(h_{2}S(h_{4}),h_{8})\overline{\#%
}_{\sigma }1_{H}\alpha (h_{6})\beta (h_{3}) \\
&=&\omega (h_{1},S(h_{3}),h_{5})1_{A}\overline{\#}_{\sigma }1_{H}\alpha
(h_{4})\beta (h_{2}) \\
&=&\varepsilon (h)1_{A}\overline{\#}_{\sigma }1_{H}
\end{eqnarray*}%
for all $h\in H$.
\end{proof}

\begin{remark}
\label{actiune slaba si cociclu pentru cleft}(1) Unlike Theorem 27 from \cite%
{Balan07}, the above theorem does not request the bijectivity of the
antipode. Notice however the presence of the element $\beta $ and of the
twist $\mathbf{f}$ in the formulas giving the weak action, the cocycle $%
\sigma $ and its inverse, making all these much more complicated than in the
Hopf algebra case.
\end{remark}

Putting together Theorem \ref{doiTakBlattnerMontg crossed} and Theorem \ref%
{iso 2 crossed prod}, we can see how two cleaving systems are related for a
given cleft extension:

\begin{corollary}
Let $H$ be a coquasi-Hopf algebra and $B\subseteq A$ a cleft extension, with
cleaving system $(\gamma $, $\delta )$. Then two linear maps $\gamma
^{\prime }$, $\delta ^{\prime }:H\longrightarrow A$ form another cleaving
system if and only if there is a convolution invertible map $\mathfrak{a}%
:H\longrightarrow B$, such that 
\begin{eqnarray*}
\gamma ^{\prime }(h) &=&\mathfrak{a}^{-1}(h_{1})\gamma (h_{2}) \\
\delta ^{\prime }(h) &=&\delta (h_{1})\mathfrak{a}(h_{2})
\end{eqnarray*}%
for all $h\in H$.
\end{corollary}

\begin{proof}
Given two cleaving systems $(\gamma ,\delta )$ and $(\gamma ^{\prime
},\delta ^{\prime })$, define $\mathfrak{a}:H\longrightarrow B$, $\mathfrak{a%
}(h)=\gamma ^{\prime }(h_{1})\beta (h_{2})\delta (h_{3})$ and $\mathfrak{a}%
^{-1}(h)=\gamma (h_{1})\beta (h_{2})\delta ^{\prime }(h_{3})$. Then it is
easy to see that $\mathfrak{a}(h),\mathfrak{a}^{-1}(h)\in B$ for all $h\in H$
and that $\mathfrak{a}$ and $\mathfrak{a}^{-1}$ are convolution inverse to
each other.
\end{proof}

\begin{remark}
For any Hopf algebra $H$, it is well known that $A=H$ is a cleft $H$%
-comodule algebra extension of $\Bbbk $ with cleaving system $(I_{H},S)$, in
particular Galois. For coquasi-Hopf algebras, this does not work anymore: if
it were true, it would mean that $A=\Bbbk \#_{\sigma }H$, impossible unless $%
H$ is a deformation of a Hopf algebra (see Remark 10.(4)). However, the
existent similarities between the defining formulas for the antipode and the
cleft extension suggest that maybe it could be possible in the future to
find an appropriate context in which $H$ can be seen as a cleft extension of 
$\Bbbk $, for any coquasi-Hopf algebra $H$.
\end{remark}

\section{Appendix: Computing crossed products for coquasi-Hopf algebras of
low dimension}

In this Section, we shall describe explicitly all cleft extensions for the
coquasi-Hopf algebras of dimension two and three constructed in \cite%
{Albuquerque99a}. These are the smallest and simplest known examples of
coquasi-Hopf algebras. However, we shall see that characterizing cleft
extensions is not simple, even if we work only with generators and
relations, and the amount of data increases even when passing from dimension 
$2$ to dimension $3$. The method we use is inspired from \cite{Doi95}.

Let $\Bbbk $ be a field of characteristic different from $2$. Following \cite%
{Albuquerque99a} or \cite{Etingof04} (where the dual case was considered),
the $2$-dimensional coquasi-Hopf algebra $H(2)$ is generated by a grouplike
element $x$ such that $x^{2}=1$. It is isomorphic as an algebra and as a
coalgebra with $\Bbbk \lbrack C_{2}]$, but with associator given by $\omega
(x,x,x)=-1$ and trivial elsewhere. It is not twist equivalent to a Hopf
algebra, and any $2$-dimensional coquasi-Hopf algebra is known to be twist
equivalent either to $\Bbbk \lbrack C_{2}]$ or to $H(2)$. The antipode is
the identity, the linear map $\beta $ is trivial $\beta (1)=\beta (x)=1$,
but the map $\alpha $ is not: $\alpha (1)=1$, $\alpha (x)=-1$.

Let $A$ be a right $H(2)$-comodule algebra and denote $B=A^{coH(2)}$.

\begin{lemma}
The extension $B\subseteq A$ is $H(2)$-cleft if and only if there exist
elements $a$, $b\in A$ such that%
\begin{eqnarray*}
ab &=&1_{A},\quad ba=-1_{A} \\
\rho _{A}(a) &=&a\otimes x,\quad \rho _{A}(b)=b\otimes x
\end{eqnarray*}

In this case, the following hold:

\begin{enumerate}
\item The maps $\gamma :H(2)\longrightarrow A$, $\gamma (1)=1_{A}$, $\gamma
(x)=a$ and $\delta :H(2)\longrightarrow A$, $\delta (1)=1_{A}$, $\delta
(x)=b $ form the cleaving system;

\item $A$ is free as left $B$-module with basis $\left\{ 1_{A},a\right\} $;

\item The elements $c=a^{2}$, $d=-b^{2}$ are invertible in $B$ and $c^{-1}=d$%
;

\item The $H(2)$-weak action on $B$ and the cocycle corresponding to the
crossed product structure are 
\begin{eqnarray*}
1\cdot e &=&e,\quad x\cdot e=aeb \\
\sigma (x,x) &=&c\text{ and trivial elsewhere}
\end{eqnarray*}%
for all $e\in B$.
\end{enumerate}
\end{lemma}

\begin{proof}
(1) For the first statement, it is enough to take $a=\gamma (x)$ and $%
b=\delta (x)$. Conversely, for any elements $a$ and $b$ with such
properties, the maps given by (1) define the cleaving system.

(2) Follows from the normal basis property (Theorem \ref{doiTakBlattnerMontg
crossed}).

(3) We have $\rho _{A}(c)=\rho _{A}(a^{2})=a^{2}\otimes x^{2}=c\otimes 1_{H}$%
. In the same way it follows that $d$ is a coinvariant element. The second
relation is easy to get from the properties of $a$ and $b$.

(4) It is enough to use Remark \ref{actiune slaba si cociclu pentru cleft}.
The convolution inverse of $\sigma $ is given by $\sigma ^{-1}(x,x)=d$. In
particular, notice that $\rho _{A}(aeb)=ae_{0}b\otimes xe_{1}x=aeb\otimes
1_{H}$, for any $e\in B$.
\end{proof}

As $H(2)$ has the coalgebra structure of $\Bbbk \left[ C_{2}\right] $, it is
natural that $A$ admits a $C_{2}$-grading: $A=A_{1}\oplus A_{x}$, with $%
A_{1}=B$. For a cleft extension, notice that the grading is strong, by Prop.
11 and Thm. 27 from \cite{Balan07}, and $A_{x}$ is free cyclic left $B$%
-module.

Let $B\subseteq A$ a cleft extension and define $\mathcal{F}%
:B\longrightarrow B$, $\mathcal{F}(e)=aeb$, for all $e\in B$, where $a,b\in
A $ are the elements given by the previous Lemma. Then by the above, $%
\mathcal{F}$ is a linear endomorphism of $B$. From the commutation relations
for $a$ and $b$ it follows that $\mathcal{F}$ is even an algebra morphism.
Statement (3) of the Lemma implies that the square of $\mathcal{F}$ is
inner: $\mathcal{F}^{2}(e)=cec^{-1}$, for all $e\in B$ and that $\mathcal{F}%
(c)=-c$. We shall see now that this is precisely what we need to built a
crossed product by $H(2)$.

\begin{proposition}
\label{cleft peste H(2)}Let $B$ be an associative algebra. For each $%
\mathcal{F}\in End(B)$ and invertible element $c\in B$, define the following:%
\begin{eqnarray*}
1\cdot e &=&e,\quad x\cdot e=\mathcal{F}(e) \\
\sigma (x,x) &=&c\text{\ and }1\text{ elsewhere}
\end{eqnarray*}%
Then $(B,\cdot ,\sigma )$ form a crossed system if and only if

\begin{enumerate}
\item $\mathcal{F}$ is an algebra morphism;

\item $\mathcal{F}^{2}(e)=cec^{-1}$, $\forall $ $e\in B$;

\item $\mathcal{F}(c)=-c$.
\end{enumerate}
\end{proposition}

\begin{proof}
Notice first that $\sigma $ is convolution invertible regardless the above
conditions, with inverse $\sigma ^{-1}(x,x)=c^{-1}$. It is easy to see that (%
\ref{weak action}) is equivalent with $\mathcal{F}$ being an algebra
endomorphism, relation (\ref{asociat}) with property (2) and (\ref{cocycle})
is equivalent with (3).
\end{proof}

For an associative algebra $B$, we call the pair $(\mathcal{F},c)$ (where $%
\mathcal{F}\in End_{\Bbbk }(B)$ and $c\in U(B)$) a cleft $H(2)$-datum for $B$
if conditions (1)-(3) from the previous proposition are fulfilled. The
crossed product $B\overline{\#}_{\sigma }H(2)$ will be denoted by $(\dfrac{%
\mathcal{F},c}{B})$. In particular, the morphism $B\longrightarrow B%
\overline{\#}_{\sigma }H(2)$, $e\longrightarrow e\otimes 1_{H}$ is injective
and preserves multiplication (can be seen as a comodule algebra morphism if
we endow $B$ with the trivial comodule structure). Via this morphism, we
shall be able to identify elements like $e\otimes 1_{H}$ with $e$, for any $%
e\in B$. If we also denote $a=1_{B}\overline{\#}_{\sigma }x$ and $%
b=-c^{-1}a=-c^{-1}\overline{\#}_{\sigma }x$, then $\left\{ 1_{B}\otimes
1_{H(2)},a\right\} $ is a left $B$-basis for the crossed product with
relations $a^{2}=c$, $b^{2}=-c^{-1}$, $ab=1_{B}\overline{\#}_{\sigma }1$, $%
ba=-1_{B}\overline{\#}_{\sigma }1$. The comodule structure is as follows: $%
\rho (a)=a\otimes x$ and $\rho (b)=b\otimes x$.

Hence we have obtained the following:

\begin{corollary}
Any cleft $H(2)$-extension $B\subseteq A$ is of the form $\left( \dfrac{%
\mathcal{F},c}{B}\right) $.
\end{corollary}

\begin{remark}
\begin{enumerate}
\item If $B=\mathbb{K\supseteq \Bbbk }$ is a field extension, then
Proposition \ref{cleft peste H(2)} implies $\mathcal{F}^{2}=I_{\mathbb{K}}$
and $\mathcal{F}(c)=-c$, $c\in \mathbb{K}\setminus \left\{ 0\right\} $. In
particular, $\mathcal{F}(e)=e^{(1)}-e^{(2)}c$ where $e=e^{(1)}+e^{(2)}c$ is
the decomposition of an element $e\in \mathbb{K}$ over the subfield $\mathbb{%
K}^{\mathcal{F}}$ fixed by $\mathcal{F}$ ($e^{(1)},e^{(2)}\in \mathbb{K}^{%
\mathcal{F}}$). For all $c\neq 1_{\mathbb{K}}$, the resulting crossed
product (as an algebra in the monoidal category of comodules) is a $\mathbb{K%
}$-vector space, with basis $\left\{ 1_{\mathbb{K}}\otimes 1,1_{\mathbb{K}%
}\otimes x\right\} $, multiplication (nonassociative, noncommutative) $%
\left( \lambda \otimes x\right) (\kappa \otimes x)=\lambda \mathcal{F}%
(\kappa )c\otimes 1$ and neutral element $1_{\mathbb{K}}\otimes 1$.

\item For $B$ a $\Bbbk $-finite dimensional central simple algebra,
conditions (1) and (2) of the previous Proposition imply that $\mathcal{F}$
is an algebra automorphism, so there is an invertible element $\widetilde{c}%
\in B$ such that $\mathcal{F}(e)=\widetilde{c}e\widetilde{c}^{-1}$, for all $%
e\in B$. In particular, it means that $c^{-1}\widetilde{c}^{2}\in Z(B)=\Bbbk 
$, therefore $c^{-1}\widetilde{c}^{2}$ is a nonzero scalar. But condition
(3) implies $\widetilde{c}c=-c\widetilde{c}$, contradiction with $char(\Bbbk
)\neq 2$. Hence for a central simple finite dimensional algebra there are no 
$H(2)$-cleft extensions.
\end{enumerate}
\end{remark}

Let see now when two such crossed products are isomophic.

\begin{proposition}
Let $\left( \mathcal{F},c\right) $ and $\left( \mathcal{F}^{\prime
},c^{\prime }\right) $ be two $H(2)$-data. Then $\left( \dfrac{\mathcal{F},c%
}{B}\right) \simeq \left( \dfrac{\mathcal{F}^{\prime },c^{\prime }}{B}%
\right) $ as right $H(2)$-comodule algebras and left $B$-modules if and only
if it exists an invertible element $s\in B$ such that $c^{\prime }=s^{-1}%
\mathcal{F}(s^{-1})c$ and $\mathcal{F}^{\prime }(e)=s^{-1}\mathcal{F}(e)s$,
for all $e\in B$.
\end{proposition}

\begin{proof}
It is a direct application of Theorem \ref{iso 2 crossed prod}. Let $\theta
:\left( \dfrac{\mathcal{F},c}{B}\right) \longrightarrow \left( \dfrac{%
\mathcal{F}^{\prime },c^{\prime }}{B}\right) $ be such an isomorphism. Then $%
\theta (e\overline{\#}_{\sigma }h)=e\mathfrak{a}(h_{1})\overline{\#}_{\sigma
^{\prime }}h_{2}$, for all $e\in B$ and $h\in H(2)$, where $\mathfrak{a}%
:H(2)\longrightarrow B$ is convolution invertible. As $\mathfrak{y}(1_{B}%
\overline{\#}_{\sigma }1)=1_{B}\overline{\#}_{\sigma ^{\prime }}1$, it
follows that $\mathfrak{a}(1_{B})=1$. Denote $\mathfrak{a}(x)=s$. Then $s$
is a unit because $\mathfrak{a}$ is convolution invertible. Moreover, by
simply applying Theorem \ref{iso 2 crossed prod}.(2) and (3) we obtain the
relations $\mathcal{F}^{\prime }(e)=s^{-1}\mathcal{F}(e)s$ and $c^{\prime
}=s^{-1}\mathcal{F}(s^{-1})c$. Conversely for $s\in B$ invertible with the
above properties, it is easy to check that the mapping $\mathfrak{a}%
:H(2)\longrightarrow B$, $\mathfrak{a}(1)=1$, $\mathfrak{a}(x)=s$ fulfills
the required conditions for the existence of a crossed product isomorphism.
\end{proof}

We consider now another example of a coquasi-Hopf algebra, but this time of
dimension $3$. Start with a field $\Bbbk $ containing a root $q\neq 1$ of
order $3$ of the unit. We denote by $H(3)$ the coquasi-bialgebra of basis $%
\left\{ 1,x,x^{2}\right\} $ with algebra and coalgebra structures as for $%
\Bbbk \lbrack C_{3}]$, but with cocycle $\omega $ given by the formulas:%
\begin{eqnarray*}
\omega (x,x^{2},x) &=&\omega (x^{2},x,x)=\omega (x^{2},x^{2},x)=q^{-1} \\
\omega (x,x^{2},x^{2}) &=&\omega (x^{2},x,x^{2})=\omega (x^{2},x^{2},x^{2})=q
\end{eqnarray*}%
and trivial in rest. Then in \cite{Albuquerque99a} it is shown that $\omega $
is not a coboundary (cannot be obtained from a twist). The antipode is the
same as for the group algebra, the linear map $\alpha $ is trivial, but $%
\beta $ is not:%
\begin{equation*}
\beta (1)=1,\beta (x)=q,\beta (x^{2})=q^{-1}
\end{equation*}

Consider $A$ a right $H(3)$-comodule algebra and $B=A^{coH(3)}$. We obtain
then the following results (that we present without proof, as they follow
the same arguments as for dimension $2$):

\begin{lemma}
The extension $B\subseteq A$ is $H(3)$-cleft if and only if there exist
elements $a$, $b$, $c$, $d\in A$ such that%
\begin{eqnarray*}
ac &=&q^{-1}1_{A},\quad ca=1_{A} \\
bd &=&q1_{A},\quad db=1_{A} \\
\rho _{A}(a) &=&a\otimes x,\quad \rho _{A}(c)=c\otimes x^{2} \\
\rho _{A}(b) &=&b\otimes x^{2},\quad \rho _{A}(d)=d\otimes x
\end{eqnarray*}

In this case, the following hold:

\begin{enumerate}
\item The maps $\gamma :H(3)\longrightarrow A$, $\gamma (1)=1_{A}$, $\gamma
(x)=a$, $\gamma (x^{2})=b$ and $\delta :H(3)\longrightarrow A$, $\delta
(1)=1_{A}$, $\delta (x)=c$, $\delta (x^{2})=d$ provide the cleft extension;

\item $A$ is free left $B$-module with basis $\{1_{A},a,b\}$;

\item If we denote 
\begin{eqnarray*}
u^{(1)} &=&(a^{2})dq^{-1},\quad u^{(2)}=(b^{2})cq \\
v^{(1)} &=&ab,\quad v^{(2)}=ba
\end{eqnarray*}%
Then these are invertible $B$-elements, with inverses 
\begin{eqnarray*}
u^{(1)-1} &=&b(c^{2})q^{-1},\quad u^{(2)-1}=a(d^{2})q \\
v^{(1)-1} &=&dcq,\quad v^{(2)-1}=cdq^{-1}
\end{eqnarray*}

\item The $H(3)$-weak action on $B$ and the cocycle are given by 
\begin{equation*}
1\cdot e=e,\quad x\cdot e=aecq,\quad x^{2}\cdot e=bedq^{-1}
\end{equation*}%
for all $e\in B$, respectively 
\begin{equation*}
\begin{tabular}{|c|c|c|c|}
\hline
$\sigma (-,-)$ & $1$ & $x$ & $x^{2}$ \\ \hline
$1$ & $1$ & $1$ & $1$ \\ \hline
$x$ & $1$ & $u^{(1)}$ & $v^{(1)}$ \\ \hline
$x^{2}$ & $1$ & $v^{(2)}$ & $u^{(2)}$ \\ \hline
\end{tabular}%
\end{equation*}
\end{enumerate}
\end{lemma}

\begin{proposition}
\label{cleft peste H(3)}Let $B$ be an associative algebra. For any $\mathcal{%
F},\mathcal{G}\in End_{\Bbbk }(B)$ and $u^{(1)}$, $u^{(2)}$, $v^{(1)}$, $%
v^{(2)}\in U(B)$, define an $H(3)$-weak action on $B$ by 
\begin{equation*}
1\cdot e=e,\quad x\cdot e=\mathcal{F}(e),\quad x^{2}\cdot e=\mathcal{G}(e)
\end{equation*}%
Consider also the linear map $\sigma :H(3)\otimes H(3)\longrightarrow B$,
given by 
\begin{equation*}
\begin{tabular}{|c|c|c|c|}
\hline
$\sigma (-,-)$ & $1$ & $x$ & $x^{2}$ \\ \hline
$1$ & $1$ & $1$ & $1$ \\ \hline
$x$ & $1$ & $u^{(1)}$ & $v^{(1)}$ \\ \hline
$x^{2}$ & $1$ & $v^{(2)}$ & $u^{(2)}$ \\ \hline
\end{tabular}%
\end{equation*}%
Then $(B,\cdot ,\sigma )$ form a crossed system if and only if

\begin{enumerate}
\item $\mathcal{F}$ and $\mathcal{G}$ are algebra endomorphisms;

\item The composition rules for these two endomorphisms are%
\begin{equation}
\begin{tabular}{|c|c|c|}
\hline
$\circ $ & $\mathcal{F}$ & $\mathcal{G}$ \\ \hline
$\mathcal{F}$ & $u^{(1)}\mathcal{G}(-)u^{(1)-1}$ & $v^{(1)}(-)v^{(1)-1}$ \\ 
\hline
$\mathcal{G}$ & $v^{(2)}(-)v^{(2)-1}$ & $u^{(2)}\mathcal{F}(-)u^{(2)-1}$ \\ 
\hline
\end{tabular}
\label{tabla pt F G}
\end{equation}

\item The actions of $\mathcal{F}$ and $\mathcal{G}$ on the above invertible
elements are%
\begin{equation*}
\begin{tabular}{|c|c|c|c|c|}
\hline
& $u^{(1)}$ & $u^{(2)}$ & $v^{(1)}$ & $v^{(2)}$ \\ \hline
$\mathcal{F}$ & $u^{(1)}v^{(2)}v^{(1)-1}$ & $v^{(1)}u^{(1)-1}q^{-1}$ & $%
u^{(1)}u^{(2)}$ & $v^{(1)}q$ \\ \hline
$\mathcal{G}$ & $v^{(2)}u^{(2)-1}q$ & $u^{(2)}v^{(1)}v^{(2)-1}q^{-1}$ & $%
v^{(2)}q^{-1}$ & $u^{(2)}u^{(1)}q$ \\ \hline
\end{tabular}%
\end{equation*}
\end{enumerate}
\end{proposition}

Let $B$ be again an associative algebra. We shall call $(\mathcal{F},%
\mathcal{G},u^{(1)},u^{(2)},v^{(1)},v^{(2)})$ (where $\mathcal{F},\mathcal{G}%
\in End_{\Bbbk }(B)$ and $u^{(1)}$, $u^{(2)}$, $v^{(1)}$, $v^{(2)}\in U(B)$)
a cleft $H(3)$-datum for $B$ if conditions (1)-(3) from the previous
Proposition are fulfilled. The resulting crossed product $B\overline{\#}%
_{\sigma }H(3)$ will be denoted by $(\dfrac{\mathcal{F},\mathcal{G}%
,u^{(1)},u^{(2)},v^{(1)},v^{(2)}}{B})$. In particular, the morphism $%
B\longrightarrow B\overline{\#}_{\sigma }H(3)$, $e\longrightarrow e\otimes
1_{H}$ is injective and preserves multiplication. Via this morphism, we may
identify elements $e\otimes 1_{H}$ with $e$, for $e\in B$. If denote $a=1_{B}%
\overline{\#}_{\sigma }x$, $b=1\overline{\#}_{\sigma }x^{2}$, $c=v^{(2)-1}%
\overline{\#}_{\sigma }x^{2}$, $d=v^{(1)-1}\overline{\#}_{\sigma }x$, then $%
\left\{ 1_{B}\otimes 1_{H(2)},a,b\right\} $ forms a left $B$-basis for $(%
\dfrac{\mathcal{F},\mathcal{G},u^{(1)},u^{(2)},v^{(1)},v^{(2)}}{B})$ with
multiplication table%
\begin{equation*}
\begin{tabular}{|c|c|c|c|c|}
\hline
& $a$ & $b$ & $c$ & $d$ \\ \hline
$a$ & $u^{(1)}\overline{\#}_{\sigma }x^{2}$ & $v^{(1)}\overline{\#}_{\sigma
}1_{H(3)}$ & $q^{-1}(1_{B}\overline{\#}_{\sigma }1_{H(3)})$ & $u^{(2)-1}%
\overline{\#}_{\sigma }x^{2}$ \\ \hline
$b$ & $v^{(2)}\overline{\#}_{\sigma }1_{H(3)}$ & $u^{(2)}\overline{\#}%
_{\sigma }x$ & $q^{-1}(u^{(1)-1}\overline{\#}_{\sigma }x)$ & $q(1_{B}%
\overline{\#}_{\sigma }1_{H(3)})$ \\ \hline
$c$ & $1_{B}\overline{\#}_{\sigma }1_{H(3)}$ & $v^{(2)-1}u^{(2)}\overline{\#}%
_{\sigma }x$ & $q^{-1}(v^{(2)-1}u^{(1)-1}\overline{\#}_{\sigma }x)$ & $%
q(v^{(2)-1}\overline{\#}_{\sigma }1_{H(3)})$ \\ \hline
$d$ & $v^{(1)-1}u^{(1)}\overline{\#}_{\sigma }1_{H(3)}$ & $1_{B}\overline{\#}%
_{\sigma }1_{H(3)}$ & $q(v^{(1)-1}\overline{\#}_{\sigma }1_{H(3)})$ & $%
v^{(1)-1}u^{(2)-1}\overline{\#}_{\sigma }x^{2}$ \\ \hline
\end{tabular}%
\end{equation*}%
The comodule structure is given by $\rho (a)=a\otimes x$, $\rho (b)=b\otimes
x^{2}$, $\rho (c)=c\otimes x^{2}$, $\rho (d)=d\otimes x$.

We have then obtained the following:

\begin{corollary}
Any $H(3)$-cleft extension $B\subseteq A$ is isomorphic to a crossed product 
$\left( \dfrac{\mathcal{F},\mathcal{G},u^{(1)},u^{(2)},v^{(1)},v^{(2)}}{B}%
\right) $.
\end{corollary}

\begin{remark}
Relations (\ref{tabla pt F G}) imply that $\mathcal{F}$ and $\mathcal{G}$
are algebra automorphisms, with $\mathcal{FG}$, $\mathcal{GF}$, $\mathcal{F}%
^{3}$ and $\mathcal{G}^{3}$ inner. In particular, for $B$ commutative, it
follows that $\mathcal{F}^{3}=I_{\mathbb{B}}$ and $\mathcal{F}^{2}=\mathcal{G%
}$.
\end{remark}

Finally we shall see when two such crossed products are isomorphic:

\begin{proposition}
Let $\left( \mathcal{F},\mathcal{G},u^{(1)},u^{(2)},v^{(1)},v^{(2)}\right) $
and $\left( \mathcal{F}^{\prime },\mathcal{G}^{\prime },u^{\prime
(1)},u^{\prime (2)},v^{\prime (1)},v^{\prime (2)}\right) $ two $H(3)$-data.
Then 
\begin{equation*}
\left( \dfrac{\mathcal{F},\mathcal{G},u^{(1)},u^{(2)},v^{(1)},v^{(2)}}{B}%
\right) \simeq \left( \dfrac{\mathcal{F}^{\prime },\mathcal{G}^{\prime
},u^{\prime (1)},u^{\prime (2)},v^{\prime (1)},v^{\prime (2)}}{B}\right)
\end{equation*}%
as right $H(3)$-comodule algebras and left $B$-modules if and only if there
exist invertible elements $s^{(1)}$, $s^{(2)}\in B$ such that 
\begin{eqnarray*}
\mathcal{F}^{\prime }(e) &=&s^{(1)-1}\mathcal{F}(e)s^{(1)} \\
\mathcal{G}^{\prime }(e) &=&s^{(2)-1}\mathcal{G}(e)s^{(2)}
\end{eqnarray*}%
for all $e\in B$, and 
\begin{eqnarray*}
u^{\prime (1)} &=&s^{(1)-1}\mathcal{F}(s^{(1)-1})u^{(1)}s^{(2)} \\
u^{\prime (2)} &=&s^{(2)-1}\mathcal{G}(s^{(2)-1})u^{(2)}s^{(1)} \\
v^{\prime (1)} &=&s^{(1)-1}\mathcal{F}(s^{(2)-1})v^{(1)} \\
v^{\prime (2)} &=&s^{(2)-1}\mathcal{G}(s^{(2)-1})v^{(2)}
\end{eqnarray*}
\end{proposition}

\begin{proof}
Denote $\mathfrak{a}(x)=s^{(1)}$ and $\mathfrak{a}(x^{2})=s^{(2)}$, where $%
\mathfrak{a}:H(3)\longrightarrow B$ is the convolution invertible map given
by Theorem \ref{iso 2 crossed prod}.
\end{proof}

\begin{acknowledgement}
The author is grateful to R. Wisbauer and J. Abuhlail for their invitation to submit a paper for this special issue of AJSE. Also she would like to thank the referees for many valuable suggestions which improved the first version of this paper.
\end{acknowledgement}%


\begin{thebibliography}{99}
\bibitem{Balan07} A. Balan, \textquotedblleft Galois {E}xtensions for {C}%
oquasi-{H}opf algebras\textquotedblright , submitted to \textit{Comm. Algebra%
}, available on arXiv:0804.3046, 2008.

\bibitem{Sweedler68} M. Sweedler, \textquotedblleft Cohomology of {A}lgebras
over {H}opf {A}lgebras\textquotedblright , \textit{Trans. AMS}, \textbf{133}
(1968), pp.205--239.

\bibitem{Blattner86} R. J. Blattner, M.~Cohen, and S.~Montgomery,
\textquotedblleft Crossed {P}roducts and {I}nner {A}ctions of {H}opf {A}%
lgebras\textquotedblright , \textit{Trans. AMS}, \textbf{298 }(1986), pp.
671--711.

\bibitem{Blattner89} R.~J. Blattner and S.~Montgomery, \textquotedblleft
Crossed {P}roducts and {G}alois {E}xtensions of {H}opf {A}%
lgebras\textquotedblright , \textit{Pacific J. Math.}, \textbf{137 }(1989),
pp. 37--54.

\bibitem{Doi89a} Y.~Doi, \textquotedblleft Equivalent {C}rossed {P}roducts
for a {H}opf {A}lgebra\textquotedblright , \textit{Comm. Algebra}, \textbf{17%
}3 (1989), pp. 3053--3085.

\bibitem{Doi86} Y.~Doi and M.~Takeuchi. \textquotedblleft Cleft {C}omodule {A%
}lgebras for a {B}ialgebra\textquotedblright , \textit{Comm. Algebra}, 
\textbf{14} (1986), pp. 801--818.

\bibitem{Schauenburg03} P.~Schauenburg, \textquotedblleft Actions of {M}%
onoidal {C}ategories and {G}eneralized {H}opf {S}mash {P}roducts%
\textquotedblright , \textit{J. Algebra}, \textbf{270} (2003), pp. 521--563.

\bibitem{Panaite04} F.~Panaite and F.~Van~Oystaeyen, \textquotedblleft Quasi-%
{H}opf {A}lgebras and {R}epresentations of {O}ctonions and other {Q}%
uasialgebras\textquotedblright , \textit{J. Math. Phys.}, \textbf{45 }%
(2004), pp. 3912--3929.

\bibitem{Hausser99} F.~Hausser and F.~Nill, \textquotedblleft Diagonal {C}%
rossed {P}roducts by {D}uals of {Q}uasi-{Q}uantum {G}roups\textquotedblright
, \textit{Rev. Math. Phys.}, \textbf{11 }(1999), pp. 553--629.

\bibitem{Bulacu02} D.~Bulacu and S.~Caenepeel, \textquotedblleft Two-{S}ided
({T}wo-{C}osided) {H}opf {M}odules and {D}oi-{H}opf {M}odules for {Q}uasi-{H}%
opf {A}lgebras\textquotedblright , \textit{J. Algebra}, \textbf{270 }(2003),
pp. 55--95.

\bibitem{Bohm07} G.~B{\"{o}}hm and J.~Vercruysse, \textquotedblleft Morita {T%
}heory for {C}oring {E}xtensions and {C}left {B}icomodules\textquotedblright
, \textit{Adv. Math.}, \textbf{209} (2007), pp. 611--648.

\bibitem{Albuquerque99a} H.~Albuquerque and S.~Majid, \textquotedblleft $%
\mathbb{Z}
_{n}$-{Q}uasialgebras\textquotedblright , in \emph{Matrices and Group
Representations (Coimbra, 1998)}, Textos Math. Ser. B 19, Coimbra, Univ.
Coimbra, pp. 57--64, 1999.

\bibitem{Majid95} S.~Majid, \textit{Foundations of Quantum Group Theory,}
Cambridge, Cambridge Univ. Press, 1995.

\bibitem{Kassel95} C.~Kassel, \textit{Quantum Groups}, Graduate Texts in
Math. 155, New York, Springer Verlag, 1995.

\bibitem{Panaite97Stefan} F.~Panaite and D.~\c{S}tefan, \textquotedblleft
When is the {C}ategory of {C}omodules a {B}raided {T}ensor {C}%
ategory?\textquotedblright , \textit{Rev. Roum. Math. Pures Appl.}, \textbf{%
42 }(1997), pp. 107--119.

\bibitem{Bulacu99} D.~Bulacu, \textquotedblleft On the {A}ntipode of {S}emi-{%
H}opf {A}lgebras and {B}raided {S}emi-{H}opf {A}lgebras\textquotedblright , 
\textit{Rev. Roum. Math. Pures Appl.}, \textbf{44 }(1999), pp. 329--340.

\bibitem{Bulacu02co} D.~Bulacu and E.~Nauwelaerts, \textquotedblleft Dual {Q}%
uasi-{H}opf {A}lgebra {C}oactions, {S}mash {C}oproducts and {R}elative {H}%
opf {M}odules\textquotedblright , \textit{Rev. Roum. Math. Pures Appl.}, 
\textbf{47 }(2002), pp. 415--443.

\bibitem{Bulacu00} D.~Bulacu and E.~Nauwelaerts, \textquotedblleft Relative {%
H}opf {M}odules for ({D}ual) {Q}uasi-{H}opf {A}lgebras\textquotedblright , 
\textit{J. Algebra}, \textbf{229 }(2000), pp. 632--659.

\bibitem{Albuquerque99} H.~Albuquerque and S.~Majid, \textquotedblleft
Quasialgebra {S}tructure of the {O}ctonions\textquotedblright , \textit{J.
Algebra}, \textbf{220 }(1999), pp. 188--224.

\bibitem{Doi94} Y.~Doi, \textquotedblleft Generalized {S}mash {P}roducts and 
{M}orita {C}ontexts for {A}rbitrary {H}opf {A}lgebras\textquotedblright , in 
\textit{Advances in Hopf Algebras}, edit. J.~Bergen, Lect. Notes Pure Appl.
Math. 158, New York, Marcel Dekker, pp. 39--53, 1994.

\bibitem{Dascalescu00} S.~D\u{a}sc\u{a}lescu, C.~N\u{a}st\u{a}sescu, and 
\c{S}. Raianu, \textit{Hopf Algebras: An Introduction}, Pure and Applied
Math. 235, New York, Marcel Dekker, 2001.

\bibitem{Montgomery93} S.~Montgomery, \textit{Hopf Algebras and Their
Actions on Rings}, {CBMS 82, }Providence, RI, AMS, 1993.

\bibitem{Maclane71} S.~MacLane, \textit{Categories for the Working
Mathematicians}, Graduate Texts Math. 5, New York, Springer-Verlag, 1971.

\bibitem{Drinfeld90} V.~G. Drinfeld, \textquotedblleft Quasi-{H}opf {A}%
lgebras\textquotedblright , \textit{Leningrad Math. J.}, \textbf{1} (1990),
pp. 1419--1457.

\bibitem{Bohm07a} G.~B{\"{o}}hm, \textquotedblleft Galois {E}xtensions over {%
C}ommutative and {N}on-commutative {B}ase\textquotedblright , \textit{New
Techniques in Hopf Algebras and Graded Ring Theory}, edit. S.~Caenepeel and
F.~Van~Oystaeyen, Koninklijke Vlaamse Academie van Belgi\"{e} voor
Wetenschappen en Kunsten, pp. 9--34, 2007, available at arXiv:math/0701064.

\bibitem{Schauenburg04} P.~Schauenburg, \textquotedblleft Hopf-{G}alois and {%
B}i-{G}alois {E}xtensions\textquotedblright , in \textit{Galois Theory, Hopf
Algebras, and Semiabelian Categories}, edit. G.~Janelidze, B.~Pareigis, and
W.~Tholen, Fields Inst. Commun. 43, AMS, 2004.

\bibitem{Doi95} Y.~Doi and M.~Takeuchi, \textquotedblleft Quaternion {A}%
lgebras and {H}opf {C}rossed {P}roducts\textquotedblright , \textit{Comm.
Algebra}, \textbf{23 }(1995), pp. 3291--3325.

\bibitem{Etingof04} P.~Etingof and S.~Gelaki, \textquotedblleft Finite {D}%
imensional {Q}uasi-{H}opf {A}lgebras with {R}adical of {C}odimension {%
2\textquotedblright , }\textit{Math. Res. Lett.}, \textbf{11 }(2004), pp.
685--696.
\end{thebibliography}
\end{document}